\documentclass[letterpaper,11pt]{amsart}
\usepackage{hyperref}
\usepackage{amsmath,amssymb,amsthm}
\usepackage{fullpage}

\makeatletter
\newtheorem*{rep@theorem}{\rep@title}
\newcommand{\newreptheorem}[2]{%
\newenvironment{rep#1}[1]{%
 \def\rep@title{#2 \ref{##1}}%
 \begin{rep@theorem}}%
 {\end{rep@theorem}}}
\makeatother

\newtheorem{thm}{Theorem}
\newreptheorem{thm}{Theorem}

\newtheorem{cor}{Corollary}
\newtheorem{lem}{Lemma}
\newreptheorem{lem}{Lemma}

\theoremstyle{definition}
\newtheorem{defn}{Definition}

\theoremstyle{remark}
\newtheorem{rem}{Remark}

\title{Sums of seven octahedral numbers}

\author{Zarathustra Brady}

\begin{document}
\maketitle

\begin{abstract} We show that for a large class of cubic polynomials $f$, every sufficiently large number can be written as a sum of seven positive values of $f$. As a special case, we show that every number greater than $e^{10^7}$ is a sum of seven positive octahedral numbers, where an octahedral number is a number of the form $\frac{2x^3+x}{3}$, reducing an open problem due to Pollock to a finite computation.
\end{abstract}

\section{Introduction}

In a paper from 1843, Sir Frederick Pollock \cite{pollock} conjectured (among other things) that every number can be written as a sum of $5$ tetrahedral numbers, $7$ octahedral numbers, $9$ cubes, $13$ icosahedral numbers, or $21$ dodecahedral numbers. That every number can be written as a sum of $7$ octahedral numbers later came to be known as Pollock's octahedral number conjecture.

More generally, given a polynomial $f$ (with positive leading coefficient) taking integer values, such that $f$ takes at least two distinct values modulo every prime $p$, one can ask if there is a number $k$ such that every sufficiently large number can be written as a sum of $k$ positive values of $f$. After subtracting a constant from $f$, it is easy to see that such functions can be written in the form
\[
f(x) = a_n\binom{x}{n} + \cdots + a_1\binom{x}{1}
\]
with $(a_1, ..., a_n) = 1$ and $a_n > 0$.

For such a polynomial $f$, Mit$'$kin \cite{mitkin} defines $G_n(f)$ to be the smallest $k$ such that every sufficiently large number can be written as a sum of $k$ positive values of $f$. Mit$'$kin also defines $H_n(f)$ to be the smallest $k$ such that for every prime power $p^m$, every congruence class modulo $p^m$ can be written as a sum of $k$ values of $f$. It is clear that $H_n(f) \le G_n(f)$ for every $f$. Next, Mit$'$kin defines $G_n = \max_{f} G_n(f)$ and $H_n = \max_{f} H_n(f)$.

Building on results of Hua \cite{hua-waring}, Mit$'$kin \cite{mitkin} shows that for all $n$ we have $H_n = 2^n-2\left\{\frac{n}{2}\right\}$, and that for $n\ge 4$ we have $G_n = H_n$. When $n = 3$, the case of interest to us, Hua proved in 1935 that $G_3(f) \le 8$ when $f$ is odd in \cite{hua-odd}, and then in 1940 Hua extended this to all cubic polynomials $f$ in \cite{hua}, showing that $G_3 \le 8$. Mit$'$kin conjectures in \cite{mitkin} that the true value of $G_3$ is $7$. An example of a cubic polynomial $f$ with $H_3(f) = 7$ is the polynomial
\[
f(x) = \frac{2(x^3-x)}{3} + x^2,
\]
which only takes the values $0,1$ modulo $8$ (that $H_3(f) \le 7$ follows from Mit$'$kin's result).

In this paper, we show that for a large class of cubic polynomials $f$ we have $G_3(f) \le 7$. Precisely, we prove the following theorems.

\begin{repthm}{oddcubic} If $f$ is an odd cubic integer-valued polynomial taking at least two distinct values modulo every prime and not of the form $f(x) = x^3 + 6kx$, then every number $n$ greater than an effectively computable constant depending only on $f$ can be written as a sum of seven positive values of $f$.
\end{repthm}

\begin{repthm}{gencubic} Let $f(x) = a\frac{x^3-x}{6} + b\frac{x^2-x}{2} + cx$ with $a > 0$ and $\gcd(a,b,c) = 1$. Suppose there is a prime $p$ such that $f$ is surjective as a function on $\mathbb{Z}_p$, and such that $v_p(a) \le v_p(2b)$. Then every number $n$ greater than an effectively computable constant depending only on $f$ can be written as a sum of seven positive values of $f$.
\end{repthm}

The method is similar to that in Linnik's original proof of the seven cubes theorem \cite{linnik-cubes}. Like Linnik, we reduce the problem to solving a convenient congruence and representing a number by a diagonal ternary quadratic form. In order to explain the main idea, we assume that $f$ is an odd polynomial with integer coefficients and ignore complications due to the primes $2$ and $3$. If $f$ has leading coefficient $a$ we start with the identity
\begin{align*}
&f(\alpha q + x) + f(\alpha q - x) + f(\beta q + y) + f(\beta q - y) + f(\gamma q + z) + f(\gamma q - z) + f(u)\\
= &\ 6aq(\alpha x^2 + \beta y^2 + \gamma z^2) + f(u) + 2f(\alpha q) + 2f(\beta q) + 2f(\gamma q).
\end{align*}
We choose $\alpha < \beta < \gamma$ to be constants with no common factor, larger than a constant times $a$. Then $q$ is chosen such that $n$ is greater than $6f(\gamma q) + f(6aq)$ and such that $f$ is surjective modulo $6aq$. Next $u$ is chosen between $0$ and $6aq$ to solve the congruence
\[
n \equiv f(u) + 2f(\alpha q) + 2f(\beta q) + 2f(\gamma q) \pmod{6aq}.
\]
Note that this reduces to $n\equiv f(u)$ when we look modulo $q$. We are left with the problems of representing a number by the diagonal form $\alpha x^2 + \beta y^2 + \gamma z^2$, and verifying the inequalities $\alpha q > x, \beta q > y, \gamma q > z$. For the inequalities we will need $6a(\alpha q)^3$ to be large compared to $n$, so $q$ will need to be between two constant multiples of $n^{\frac{1}{3}}$.

When representing a number as a sum of seven cubes, it is convenient to take $q$ to be a prime $p$ satisfying $p\equiv 2\pmod{3}$. In the case of a general cubic polynomial $f$ there is typically only a finite set of primes modulo which $f$ is surjective, so we will instead take $q$ to be a power of a fixed prime $p$, depending on the function $f$. A complication occurs since now the set of convenient moduli $q$ is sparse on the logarithmic scale, so instead of using a single diagonal ternary quadratic form as in Linnik's argument for cubes, we must use a collection of several diagonal ternary forms to handle different ranges for the remainder of $\log(n)$ modulo $\log(p)$.

It is natural to wonder whether Watson's simplified proof of the seven cubes theorem \cite{watson-cubes} can be similarly generalized. The main idea of Watson's argument is to take advantage of the fact that if $u \equiv 0 \pmod{r}$ then $u^3 \equiv 0 \pmod{r^3}$. In order to generalize Watson's argument, we could try to take advantage of ramification: find a congruence class $t$ modulo a small prime $r$ such that $u \equiv t \pmod{r}$ implies $f(u) \equiv f(t) \pmod{r^2}$, and to try to represent $n$ by an expression such as
\[
f(qr^2 + sx) + f(qr^2 - sx) + f(qr^2 + sy) + f(qr^2 - sy) + f(qs^2 + rz) + f(qs^2 - rz) + f(u),
\]
where $s$ is a small prime chosen similarly to $r$. The next step in such a generalization would be to solve (for $q$) a congruence such as
\[
n - f(t) \equiv 2f(qs^2) \pmod{r^2}.
\]
Here, unfortunately, we tend to get stuck, since a general cubic polynomial $f$ is only surjective modulo finitely many primes $r$. The author has not been able to find any simple variation of Watson's argument which gets around this difficulty.

In addition to the above results about general cubic polynomials, we also find an explicit bound for the problem of representing a number as a sum of seven octahedral numbers, reducing Pollock's octahedral number conjecture to a finite computation.

\begin{repthm}{octahedral} Every number $n$ greater than $e^{10^7}$ is a sum of seven positive octahedral numbers, where an octahedral number is a number of the form $\frac{2x^3+x}{3}$ ($x$ an integer).
\end{repthm}

In this case we have two distinct primes $p$, namely $2$ and $3$, such that the octahedral polynomial is surjective modulo all powers of $p$. This allows us to take a shortcut in the general argument, and deal with a single ternary quadratic form, for which we have the following result.

\begin{repthm}{quadratic} If $m, m' > e^{6.6\cdot 10^6}$, $m, m' \equiv \pm 2 \pmod{5}$, $m\equiv 3\pmod{4}$, and $m' \equiv 2\pmod{4}$, then at least one of $m, m'$ can be primitively represented by the quadratic form $83x^2 + 91y^2 + 99z^2$ (a primitive representation of $m$ by a quadratic form $Q(x,y,z)$ is a triple $x,y,z$ with $\gcd(x,y,z) = 1$ and $Q(x,y,z) = m$).
\end{repthm}

The proof of this follows the argument in Linnik \cite{linnik-quaternions} with several refinements in order to reduce the bounds. The main inefficiency in Linnik's argument is his use of the divisor bound $\tau(n) \ll_{\epsilon} n^{\epsilon}$, which leads to a final bound which is doubly exponential in the discriminant of the quadratic form. By modifying Linnik's geometric argument which bounds the number of ``conjugate pairs of quaternions directed by senior forms'' in terms of the number of representations of certain binary quadratic forms as sums of three squares (see Lemma \ref{senior}), and carefully bounding the resulting number theoretic sum using a level lowering trick (see Lemma \ref{level} and Lemma \ref{sum}), we show that on average the contributions from the divisor function are at most a power of $\log(n)$, which reduces the final bounds from doubly exponential to singly exponential in the discriminant.

One additional ingredient was needed to prove the above theorem, which may be of independent interest: a version of the Siegel-Tatuzawa-Hoffstein theorem \cite{hoffstein} which applies when the characters under consideration are not necessarily primitive, and when the discriminants of the characters are allowed to be much smaller than the range in which the theorem is usually applied. This is used to give lower bounds for the number of primitive representations of possibly non-squarefree numbers as a sum of three squares.

\begin{replem}{siegel} Let $t(m)$ be the number of primitive representations of $m$ as a sum of three squares. If $m, m'$ are positive integers with different squarefree parts which are not multiples of $4$ and not congruent to $7\pmod{8}$, and if $0 \le \epsilon < 10^{-3}$, then we have
\[
\max\left(\frac{t(m)}{m^{\frac{1}{2}-\epsilon}},\ \frac{t(m')}{m'^{\frac{1}{2}-\epsilon}}\right) \ge \frac{12\epsilon}{\pi}\prod_{p\text{ \rm odd prime}} \min\left(1,p^{2\epsilon}\left(1-\frac{1}{p}\right)\right).
\]
\end{replem}

The proof of this variation of the Siegel-Tatuzawa-Hoffstein theorem uses standard techniques, but it appears that until now no one has published such a result which does not require the squarefree parts of $m,m'$ to be larger than $e^{1/\epsilon}$.

\section{Some elementary lemmata}

First we recall a useful lemma of Watson \cite{watson-cubic} on the number of values taken on by a cubic polynomial modulo a prime $p$.
\begin{lem}\label{watson}
If $p\ne 3$ is a prime and $a\not\equiv 0\mbox{ \rm (mod } p)$, then the congruence
\[
x^3+ax\equiv n\mbox{ \rm (mod } p)
\]
is solvable for exactly $\lfloor\frac{2p+1}{3}\rfloor$ congruence classes $n$.
\end{lem}
\begin{proof}
See \cite{watson-cubic}, Lemma 1.
\end{proof}

\begin{lem}\label{p-adic} Let $f(x) = a\frac{x^3-x}{6} + b\frac{x^2-x}{2} + cx$ with $\gcd(a,b,c) = 1$, and let $p$ be prime. Then $f$ is surjective as a function $\mathbb{Z}_p \rightarrow \mathbb{Z}_p$ if and only if either the $p$-adic valuation of $\gcd\left(\frac{a}{6},\frac{b}{2}\right)$ is nonzero (in particular, if $p = 2$ this means that either $a$ and $b$ are both multiples of $4$ or at least one of $a,b$ is odd), or $p = 3$ and $3 \mid b$ and $\frac{a}{6} \not\equiv c \pmod{3}$.

In particular, if $f$ is odd and not of the form $f(x) = x^3+6kx$, then there exists a prime $p$ such that $f$ is surjective as a function $\mathbb{Z}_p \rightarrow \mathbb{Z}_p$.
\end{lem}
\begin{proof} First suppose that $p$ divides both $\frac{a}{6}$ and $\frac{b}{2}$. Then we must have $p \nmid c$, so $f:\mathbb{Z}/p \rightarrow \mathbb{Z}/p$ is surjective, and $f'(x) \equiv c \not\equiv 0 \pmod{p}$ so by Hensel's Lemma $f$ is surjective onto $\mathbb{Z}_p$.

Now suppose $p=2$ and one of $a,b$ is odd. If $a+b$ is odd, then $f(2x) = \frac{4a}{3}x^3 + 2bx^2 + (2c-\frac{a}{3}-b)x \equiv x \pmod{2}$ is surjective onto $\mathbb{Z}_2$ by Hensel's Lemma. Otherwise $b$ is odd, so $f(2x+1) = \frac{4a}{3}x^3 + 2(a+b)x^2 + (2c+\frac{4a}{3}+b)x+c \equiv x+c \pmod{2}$ is surjective onto $\mathbb{Z}_2$ by Hensel's Lemma.

Now suppose $p=3$. If $a$ is not a multiple of $3$ then $f(3x) = \frac{9a}{2}x^3 + \frac{9b}{2}x^2 + (3c-\frac{3b}{2} - \frac{a}{2})x \equiv ax \pmod{3}$ is surjective onto $\mathbb{Z}_3$ by Hensel's Lemma. If $3$ divides $a,b$ and $\frac{a}{6} \not\equiv c \pmod{3}$, then $f(x) \equiv cx \pmod{3}$ is surjective onto $\mathbb{Z}/3$ and $f'(x) \equiv c - \frac{a}{6} \ne 0 \pmod{3}$ so by Hensel's Lemma $f$ is surjective onto $\mathbb{Z}_3$.

For the converse, note that if $\frac{a}{6}$ and $\frac{b}{2}$ are both $p$-integral then $f$ induces functions $\mathbb{Z}/p\rightarrow \mathbb{Z}/p$ and $\mathbb{Z}/p^2\rightarrow \mathbb{Z}/p^2$. If $p \ge 5$ then by Lemma \ref{watson} and the fact that only $\frac{p+1}{2}$ congruence classes mod $p$ are squares, the only way for $f$ to be a bijection mod $p$ is for $f(x)$ to be congruent to either a linear function of $x$, or a constant plus the cube of a linear function of $x$, and in the second case $f$ can't be surjective mod $p^2$. If $p$ is $2$ or $3$ then similar reasoning shows that $f$ can only be a bijection mod $p^2$ if it satisfies the given conditions.
\end{proof}

By the proof of the preceding Lemma we easily obtain the following Corollary.

\begin{cor}\label{octa}
If $a,b$ are any nonnegative integers, then for any integer $n$ the congruence
\[
\frac{2x^3+x}{3} \equiv n\pmod{2^u3^v}
\]
is solved by exactly $3$ congruence classes $x$ modulo $2^u3^{v+1}$.
\end{cor}

Finally we have an easy lemma about multiplicatively independent numbers.
\begin{lem}\label{power}
If $a,b > 1$ are multiplicatively independent positive integers (that is, if $a^x \ne b^y$ for all pairs of positive integers $x,y$), then for every $\epsilon > 0$ there exists a number $D_0$ such that for every $D > D_0$ there exist positive integers $x,y$ satisfying
\[
D < a^xb^y < (1+\epsilon)D.
\]
\end{lem}
\begin{proof} By a standard application of the pigeonhole principle, there exist integers $u,v$ such that $1 < a^ub^v < 1 + \epsilon$. Suppose without loss of generality that $v > 0$. Set
\[
D_0 = a^{1+|u|\left\lceil\frac{\log(a)}{\log(a^ub^v)}\right\rceil}.
\]
Then for any $D > D_0$, let $k$ be the largest integer such that $a^k \le D$, and let $l$ be the least integer such that $D < a^k(a^ub^v)^l$. In this case we necessarily have
\[
D < a^k(a^ub^v)^l \le a^ub^vD < (1+\epsilon)D,
\]
\[
1+|u|\left\lceil\frac{\log(a)}{\log(a^ub^v)}\right\rceil \le k,
\]
and
\[
0 < l \le \left\lceil\frac{\log(a)}{\log(a^ub^v)}\right\rceil.
\]
Taking $x = k + ul, y = vl$, we see that $x,y > 0$ and $D < a^xb^y < (1+\epsilon)D$.
\end{proof}

\section{Odd cubic polynomials}

We rely on the following result of Linnik \cite{linnik-quaternions}.

\begin{lem}\label{linnik}
If $\alpha,\beta,\gamma,s$ are pairwise relatively prime odd natural numbers, $s$ a prime, such that $-\alpha\beta$ is a square modulo $\gamma$, $-\alpha\gamma$ is a square modulo $\beta$, and $-\beta\gamma$ is a square modulo $\alpha$, then the quadratic form $\alpha x^2+\beta y^2+\gamma z^2$ primitively represents every sufficiently large integer $m$ such that $\left(\frac{-\alpha\beta\gamma m}{s}\right) = 1$, $m\not\equiv 0\pmod{4}$, and such that $\alpha\beta\gamma m$ can be primitively represented as a sum of three squares (here $\left(\frac{-\alpha\beta\gamma m}{s}\right)$ is the Jacobi symbol).

Furthermore, there is an effectively computable constant $C$ depending only on $\alpha,\beta,\gamma,s$ such that if $m, m'$ are two numbers as above having different squarefree parts and satisfying $C < m,m'$, then at least one of $m,m'$ can be primitively represented by the quadratic form $\alpha x^2+\beta y^2+\gamma z^2$.
\end{lem}
\begin{proof}
The first part of this Lemma is Theorem 2 from chapter II of \cite{linnik-quaternions} specialized to the case of diagonal quadratic forms, although that theorem contains the extra condition that $m$ is relatively prime to $\alpha\beta\gamma$. (The condition that $m$ is relatively prime to $\alpha\beta\gamma$ was only used in the proof of the main statement of Chapter II, Section 1, in order to get the exact count of the number of primitive representations, and without it we still get a lower bound on the number of primitive representations.)

For the second part of the Lemma, we note that the only ineffective bound in Linnik's argument is the lower bound on the number of primitive representations of $\alpha\beta\gamma m$ as a sum of three squares. By Siegel's theorem, we can find effective lower bounds for the number of primitive representations of at least one of $\alpha\beta\gamma m, \alpha\beta\gamma m'$ as a sum of three squares if $m, m'$ have different squarefree parts (see Lemma \ref{siegel} for details).

We will give a self contained proof of this lemma, with explicit constants, in the special case $\alpha = 83, \beta = 91, \gamma = 99, s = 5$ in a later section.
\end{proof}

\begin{rem} If $\alpha, \beta, \gamma$ satisfy the conditions of Lemma \ref{linnik}, then we necessarily have $\alpha \equiv \beta \equiv \gamma \pmod{4}$. To see this, suppose that $\alpha \equiv 1 \pmod{4}$. Expanding $\left(\frac{-\beta\gamma}{\alpha}\right)\left(\frac{-\alpha\gamma}{\beta}\right)\left(\frac{-\alpha\beta}{\gamma}\right) = 1\cdot 1\cdot 1 = 1$ and applying the law of quadratic reciprocity we get
\[
1 = \left(\frac{-1}{\alpha}\right)\left(\frac{\beta}{\alpha}\right)^2\left(\frac{\gamma}{\alpha}\right)^2\left(\frac{-\beta}{\gamma}\right)\left(\frac{-\gamma}{\beta}\right) = \left(\frac{-\beta}{\gamma}\right)\left(\frac{-\gamma}{\beta}\right),
\]
and this holds if and only if $\beta \equiv \gamma \equiv 1 \pmod{4}$. Thus if any of $\alpha, \beta, \gamma$ is $1 \pmod{4}$ then the other two are as well.

As a consequence of $\alpha \equiv \beta \equiv \gamma \equiv \pm 1 \pmod{4}$, we have the congruence
\[
\alpha + \beta + \gamma \equiv 3\alpha\beta\gamma \pmod{8},
\]
which can be shown by writing $\alpha = 4A\pm 1, \beta = 4B\pm 1, \gamma = 4C\pm 1$ and expanding the product on the right.
\end{rem}

\begin{thm}\label{oddcubic} If $f$ is an odd cubic integer-valued polynomial taking at least two distinct values modulo every prime and not of the form $f(x) = x^3 + 6kx$, then every number $n$ greater than an effectively computable constant depending only on $f$ can be written as a sum of seven positive values of $f$.
\end{thm}
\begin{proof} We can write $f = a\frac{x^3-x}{6}+cx$ for $a,c\in \mathbb{Z}$ with $(a,c) = 1$. Note that we have
\[
f\left(\frac{x+y}{2}\right) + f\left(\frac{x-y}{2}\right) = \frac{ax}{8}y^2 + 2f\left(\frac{x}{2}\right).
\]
Let $p$ be the smallest prime such that $f$ is surjective when considered as a function $\mathbb{Z}_p \rightarrow \mathbb{Z}_p$, as in Lemma \ref{p-adic} (so if $p \ne 2$ then $a \equiv 2 \pmod{4}$). Fix two distinct primes $s, t$ greater than or equal to $19$, not equal to $p$, and not dividing $a$ or $6c-a$. We will attempt to find, for every $n$, a solution to the following Diophantine equation:
\[
4n - 8f\left(\frac{\alpha p^{j}}{2}\right) - 8f\left(\frac{\beta p^{j}}{2}\right) - 8f\left(\frac{\gamma p^{j}}{2}\right) - 4f(u) = \frac{ap^{j}}{2}(\alpha x^2 + \beta y^2 + \gamma z^2),
\]
with $\alpha, \beta, \gamma$ coming from a fixed finite set of values relatively prime to $s, t$ and depending only on $f$ and $p$, and $u$ chosen less than $6astp^j$ to make the left hand side satisfy various congruence conditions modulo $8, a, p^{j}, s$, and $t$. If we can achieve this, then we can write
\begin{align*}
n =\; &\frac{ap^{j}}{8}\left(\alpha x^2 + \beta y^2 + \gamma z^2\right) + 2f\left(\frac{\alpha p^{j}}{2}\right) + 2f\left(\frac{\beta p^{j}}{2}\right) + 2f\left(\frac{\gamma p^{j}}{2}\right) + f(u)\\
=\; &f\left(\frac{\alpha p^{j} + x}{2}\right) + f\left(\frac{\alpha p^{j} - x}{2}\right) + f\left(\frac{\beta p^{j} + y}{2}\right) + f\left(\frac{\beta p^{j} - y}{2}\right)\\
&+ f\left(\frac{\gamma p^{j} + z}{2}\right) + f\left(\frac{\gamma p^{j} - z}{2}\right) + f(u).
\end{align*}

Using Lemma \ref{power}, choose $D_0$ so large that for every pair of distinct odd primes $\pi,\rho \le 17$, and for any $D > D_0$, there exist positive integers $x,y$ such that
\[
D < \pi^{2x}\rho^{2y} < 1.01D.
\]
In particular, for this choice of $D_0$, for any $D > D_0/\rho$ there also exist positive integers $x',y'$ such that
\[
D\rho < \pi^{2x'}\rho^{2y'} < 1.01D\rho,
\]
or equivalently,
\[
D < \pi^{2x'}\rho^{2y'-1} < 1.01D.
\]
Now we fix a finite collection of real numbers $D_i > \max(D_0,12ast)$, say $i = 1, ..., I$, such that for every $x > 0$ there exists an $1 \le i \le I$ and an integer $j$ satisfying
\[
\frac{3}{2}D_i^3+1 < \frac{x}{p^{3j}} < 2D_i^3-1.
\]
For instance, we could take $I = \lceil 11\log(p)\rceil$ and $D_i = p^{\frac{i}{I}}\max(D_0,12ast)$. To see that this works, note first that for each $i$ we have $D_i > 12$, so from
\[
1-\frac{3}{4}p^{\frac{3}{I}} > 1-\frac{3}{4}e^{\frac{3}{11}} = 0.0148... > \frac{1}{D_i^3}
\]
we see that
\[
D_i^3-\frac{3}{4}D_{i+1}^3 > 1,
\]
or equivalently
\[
\frac{3}{2}D_{i+1}^3 + 1 < 2D_i^3-1.
\]
Similarly one can check that
\[
p^3\Big(\frac{3}{2}D_1^3 + 1\Big) < 2D_I^3-1.
\]
Thus, if $j$ is chosen to be as large as possible such that $\frac{x}{p^{3j}} > \frac{3}{2}D_1^3 + 1$, then $\frac{x}{p^{3j}}$ lies in the open interval $(\frac{3}{2}D_1^3 + 1, 2D_I^3-1)$, and the open interval $(\frac{3}{2}D_1^3 + 1, 2D_I^3-1)$ is covered by the intervals $(\frac{3}{2}D_i^3+1, 2D_i^3-1)$ since each one overlaps with the next.

We now construct a family of quadratic forms $\alpha_{i}x^2 + \beta_{i}y^2 + \gamma_{i}z^2$ satisfying the conditions of Lemma \ref{linnik}, for $1 \le i \le I$, such that
\begin{align}\label{sizes}
D_i < \alpha_{i} < 1.01D_i < \beta_{i}, \gamma_{i} < 1.03D_i.
\end{align}
To do so, we take $\alpha_{i} = 3^{\rm even}5^{\rm even}, \beta_{i} = 7^{\rm even}17^{\rm odd}, \gamma_{i} = 11^{\rm even}13^{\rm odd}$, where the exponents are positive integers chosen to satisfy \eqref{sizes} (such exponents exist by the choice of $D_0$, the case when one exponent is odd follows from the comment following the definition of $D_0$). These satisfy the conditions of Lemma \ref{linnik} since
\[
\left(\frac{-13\cdot 17}{3}\right) = \left(\frac{-13\cdot 17}{5}\right) = \left(\frac{-13}{7}\right) = \left(\frac{-13}{17}\right) = \left(\frac{-17}{11}\right) = \left(\frac{-17}{13}\right) = 1.
\]

Choose $1\le i \le I$ and an integer $j$ such that
\[
\frac{3}{2}D_i^3 + 1 < \frac{8n}{ap^{3j}} < 2D_i^3 - 1.
\]
For $n$ sufficiently large, this $j$ is positive and by \eqref{sizes} satisfies the inequalities
\[
f(6astp^{j}) + ap^{2j} < n - 2f\left(\frac{\alpha_{i} p^{j}}{2}\right) - 2f\left(\frac{\beta_{i} p^{j}}{2}\right) - 2f\left(\frac{\gamma_{i} p^{j}}{2}\right) < \frac{a\alpha_{i}^3p^{3j}}{8}.
\]
From now on write $\alpha = \alpha_{i}, \beta = \beta_{i}, \gamma = \gamma_{i}$.

Write
\[
d = 4n - 8f\left(\frac{\alpha p^{j}}{2}\right) - 8f\left(\frac{\beta p^{j}}{2}\right) - 8f\left(\frac{\gamma p^{j}}{2}\right).
\]
Note that $d$ is an integer, and that $d \equiv \frac{ap^{3j}}{6} \pmod{4}$ since $\alpha \equiv \beta \equiv \gamma \equiv 1 \pmod{4}$. We want to show that we can choose $u$ such that $\frac{d-4f(u)}{\frac{1}{2}ap^{j}}$ is an integer represented by the quadratic form $\alpha x^2 + \beta y^2 + \gamma z^2$. To do so, we need to choose $u$ satisfying several congruence conditions, and we will now check that they can be satisfied.

The congruence
\[
8f(u) \equiv 2d \pmod{a}
\]
has a solution since $f(u) \equiv cu \pmod{a}$, $(a,c) = 1$, and $(8,a)\mid 2d$. Further, if $p^e\mid\mid a$, the congruence
\[
8f(u) \equiv 2d \pmod{p^{j+e}}
\]
has a solution since $f$ is surjective mod every power of $p$ and since $4\mid d$ if $p = 2$. Together, these imply that the congruence
\[
8f(u) \equiv 2d\pmod{ap^{j}}
\]
has a solution, with $u$ determined modulo $ap^j$ if $p\mid a$, and $u$ determined modulo $ap^{j+1}$ otherwise (in which case $p \le 3$).

Next we check that there is a solution to the congruence
\[
-\alpha\beta\gamma y^2 \equiv \frac{d - 4f(u)}{\frac{1}{2}ap^{j}} \pmod{s}
\]
with $y \not\equiv 0 \pmod{s}$: the number of nonzero squares modulo $s$ is $\frac{s-1}{2}$, and since $s\nmid a(6b-a)$ the number of distinct values of $f$ modulo $s$ is $\lfloor\frac{2s+1}{3}\rfloor$ by Lemma \ref{watson}, so it is enough to check that
\[
\frac{s-1}{2} + \left\lfloor\frac{2s+1}{3}\right\rfloor \ge s+1,
\]
and this clearly holds for primes $s$ greater than or equal to $19$. Similarly, for any given value among $1, -1$ we can choose $u$ modulo $t$ such that $\left(\frac{d-4f(u)}{t}\right)$ takes that value.

Finally, we need to check that we can also solve the congruence
\[
\alpha\beta\gamma\frac{d-4f(u)}{\frac{1}{2}ap^{j}} \equiv 3 \pmod{8}
\]
when $p \ne 2$, or the congruence
\[
\alpha\beta\gamma\frac{d-4f(u)}{a2^{j+1}} \equiv 1, 2 \pmod{4}
\]
when $p$ is $2$. If $p \ne 2$, then $\frac{1}{2}ap^j$ is odd and $\alpha\beta\gamma\frac{d}{\frac{1}{2}ap^{j}} \equiv 3 \pmod{4}$, so we can solve the congruence modulo $8$ since $f$ takes both even and odd values. If $p$ is $2$, then we can solve the congruence modulo $4$ since $4\mid d$, $4\mid a2^{j+1}$, and $f$ is surjective as a function from $\mathbb{Z}_2 \rightarrow \mathbb{Z}_2$, and $u$ is determined modulo $2^{j+2}$ if $2\nmid a$, or determined modulo $2^{j+e+1}$ if $2^e\mid\mid a$, $e \ge 2$.

Thus we can choose a number $u$ between $0$ and $6astp^{j}$ such that $d-4f(u)$ is a multiple of $\frac{1}{2}ap^{j}$, and such that if we write
\[
m = \frac{d-4f(u)}{\frac{1}{2}ap^{j}}
\]
then we have $\left(\frac{-\alpha\beta\gamma m}{s}\right) = 1$, $\left(\frac{m}{t}\right) = 1$ and $\frac{\alpha\beta\gamma m}{(2,p)^2}$ is primitively representable as a sum of three squares. Similarly we can choose $u'$ between $0$ and $6astp^{j}$ such that $m' = \frac{d-4f(u')}{\frac{1}{2}ap^{j}}$ has all the same properties, except $\left(\frac{m'}{t}\right) = -1$.

From the bounds on $p^{j}$, we have $p^{j} < m, m' < \alpha^3p^{2j}$. Since $\left(\frac{m}{t}\right) = 1, \left(\frac{m'}{t}\right) = -1$, $m$ and $m'$ have different squarefree parts. Thus by Lemma \ref{linnik}, for $p^{j}$ larger than an effectively computable constant at least one of $\frac{m}{(2,p)^2}, \frac{m'}{(2,p)^2}$ can be primitively represented by the quadratic form $\alpha x^2 + \beta y^2 + \gamma z^2$. Assume without loss of generality that $m$ is representable, say
\[
m = \alpha x^2 + \beta y^2 + \gamma z^2.
\]
Note that since $m \equiv 3 \pmod{4}$ if $p \ne 2$ and $4 \mid m$ if $p = 2$, we must have $x \equiv y \equiv z \equiv p \pmod{2}$. From the upper bound on $m$, we have $x,y,z < \alpha p^{j}$. Thus we can write

\begin{align*}
n =\; &f\left(\frac{\alpha p^{j} + x}{2}\right) + f\left(\frac{\alpha p^{j} - x}{2}\right) + f\left(\frac{\beta p^{j} + y}{2}\right) + f\left(\frac{\beta p^{j} - y}{2}\right)\\
&+ f\left(\frac{\gamma p^{j} + z}{2}\right) + f\left(\frac{\gamma p^{j} - z}{2}\right) + f(u).\qedhere
\end{align*}
\end{proof}

\begin{rem} For most odd cubic polynomials of the form $f(x) = a\frac{x^3-x}{6} + cx$, it is enough to exploit the identity
\[
f(x+y) + f(x-y) = axy^2 + 2f(x).
\]
There is a case where this doesn't work. Let $f(x) = \frac{x^3-4x}{3}$. Then $f(x) \equiv 0\pmod{16}$ whenever $x$ is even, and further for any $k, y$ we have
\[
f(4k+1+y) + f(4k+1-y) \equiv 0, -2, 6 \pmod{16}
\]
and
\[
f(4k+3+y) + f(4k+3-y) \equiv 0, 2, -6 \pmod{16}.
\]
Thus if $n \equiv 8 \pmod{16}$, the congruence
\[
n\equiv f(\alpha p^j+x) + f(\alpha p^j-x) + f(\beta p^j+y) + f(\beta p^j-y) + f(\gamma p^j+z) + f(\gamma p^j-z) + f(u) \pmod{16}
\]
has no solutions for $p$ odd and $\alpha \equiv \beta \equiv \gamma \equiv \pm 1 \pmod{4}$.
\end{rem}

\section{General cubic polynomials}

\begin{lem}\label{abc} Let $a,b$ be relatively prime integers, $a$ positive. There exist six primes $p_1, ..., p_6$ depending only on $a,b$ such that for any $\epsilon > 0$ there is a $D_0$ such that for all $D > D_0$ there exist $\alpha,\beta,\gamma$ supported on $p_1, ..., p_6$ (a number is supported on a set of primes if all of its prime factors lie in that set) satisfying the conditions of Lemma \ref{linnik} as well as
\[
\alpha \equiv \beta \equiv \gamma \equiv b \pmod{a}
\]
and
\[
D < \alpha < (1+\epsilon)D < \beta, \gamma < (1+\epsilon)^2D.
\]
\end{lem}
\begin{proof} Assume without loss of generality that $b$ is odd and $4$ divides $a$. Start by picking distinct $p_1, p_2, p_3$ congruent to $1$ modulo $a$. Next choose $p_4, p_5, p_6$ congruent to $b$ modulo $a$ such that all three are quadratic residues modulo each of $p_1, p_2, p_3$, such that $p_5$ is a quadratic residue modulo $p_4$, $p_6$ is a quadratic residue modulo $p_5$, and $\left(\frac{p_6}{p_4}\right) \equiv b \pmod{4}$. Then by the laws of quadratic reciprocity we have
\[
\left(\frac{-p_4p_5}{p_6}\right) = \left(\frac{-p_5p_6}{p_4}\right) = \left(\frac{-p_6p_4}{p_5}\right) = 1. 
\]
Thus for any positive integers $k_1, ..., k_6$, $\alpha = p_1^{\varphi(a)k_1}p_4^{1+\varphi(a)k_4}, \beta = p_2^{\varphi(a)k_2}p_5^{1+\varphi(a)k_5},$ and $\gamma = p_3^{\varphi(a)k_3}p_6^{1+\varphi(a)k_6}$ satisfy the conditions of Lemma \ref{linnik}. Now we apply Lemma \ref{power} to finish.
\end{proof}

\begin{thm}\label{gencubic} Let $f(x) = a\frac{x^3-x}{6} + b\frac{x^2-x}{2} + cx$ with $a > 0$ and $\gcd(a,b,c) = 1$. Suppose there is a prime $p$ such that $f$ is surjective as a function on $\mathbb{Z}_p$, and such that $v_p(a) \le v_p(2b)$. Then every number $n$ greater than an effectively computable constant depending only on $f$ can be written as a sum of seven positive values of $f$.
\end{thm}
\begin{proof} Note that we have
\[
f\left(\frac{x+y}{2}\right) + f\left(\frac{x-y}{2}\right) = \frac{ax+2b}{8}y^2 + 2f\left(\frac{x}{2}\right).
\]
Let $p$ be the smallest prime as in the theorem statement (so if $p \ne 2$ then either $\gcd(a,b) \equiv 2 \pmod{4}$ or $v_2(a) \ge v_2(4b)$). Set $g = \gcd(a,2b)$, and let $a' = \frac{a}{g}, b' = \frac{2b}{g}$. Define $a'', b''$ as follows: if $b'$ is even or $f:\mathbb{Z}_2\rightarrow \mathbb{Z}_2$ is surjective, then $a'' = a', b'' = b'$, otherwise $a'' = 4a'$, $b'' \equiv b' \pmod{2a'}$, $b''$ chosen mod $2a$ such that
\[
n-f(n)-6f\left(\frac{b''-b'}{2a'}\right) \equiv 2 \pmod{4},
\]
which we can achieve since in this case $a,b$ are even (by Lemma \ref{p-adic} and the assumption that $f:\mathbb{Z}_2\rightarrow \mathbb{Z}_2$ is not surjective) and $c$ is odd, so $n - f(n) \equiv 0\pmod{2}$ automatically.

Find primes $p_1, ..., p_6$ as in Lemma \ref{abc} applied to $a'',b''$. Let $q$ be a power of $p$ which is congruent to $1$ modulo $a''$ (this exists since by assumption $a'$ and $p$ are relatively prime). Fix two distinct primes $s, t$ congruent to $5$ modulo $6$, distinct from $p_1, ..., p_6$, not equal to $p$, and not dividing $a$. We will attempt to find, for every $n$, a solution to the following Diophantine equation:
\[
4n - 8f\left(\frac{\alpha q^j-b'}{2a'}\right) - 8f\left(\frac{\beta q^j-b'}{2a'}\right) - 8f\left(\frac{\gamma q^j-b'}{2a'}\right) - 4f(u) = \frac{gq^{j}}{2}(\alpha x^2 + \beta y^2 + \gamma z^2),
\]
with $\alpha, \beta, \gamma$ coming from a fixed finite set of values supported on $p_1, ..., p_6$ and depending only on $f$ and $p$, $x,y,z$ all of the same parity as $\frac{b''-b'}{a'}$, and $u$ chosen less than $12gstq^j$ to make the left hand side satisfy various congruence conditions modulo $8g, q^{j}, s$, and $t$. If we can achieve this, then we can write
\begin{align*}
n =\; &\frac{gq^{j}}{8}\left(\alpha x^2 + \beta y^2 + \gamma z^2\right) + 2f\left(\frac{\alpha q^j-b'}{2a'}\right) + 2f\left(\frac{\beta q^j-b'}{2a'}\right) + 2f\left(\frac{\gamma q^j-b'}{2a'}\right) + f(u)\\
=\; &f\left(\frac{\alpha q^{j}-b'}{2a'}+\frac{x}{2}\right) + f\left(\frac{\alpha q^{j}-b'}{2a'}-\frac{x}{2}\right) + f\left(\frac{\beta q^{j}-b'}{2a'}+\frac{y}{2}\right) + f\left(\frac{\beta q^{j}-b'}{2a'}-\frac{y}{2}\right)\\ &+ f\left(\frac{\gamma q^{j}-b'}{2a'}+\frac{z}{2}\right) + f\left(\frac{\gamma q^{j}-b'}{2a'}-\frac{z}{2}\right) + f(u).
\end{align*}

Write
\[
d = 4n - 8f\left(\frac{\alpha q^j-b'}{2a'}\right) - 8f\left(\frac{\beta q^j-b'}{2a'}\right) - 8f\left(\frac{\gamma q^j-b'}{2a'}\right).
\]
Note that if $\frac{b''-b'}{a'}$ is odd, then
\begin{align*}
d - \frac{gq^j}{2}(\alpha+\beta+\gamma) = 4n-4\Bigg(&f\left(\frac{\alpha q^{j}-b'}{2a'}+\frac{1}{2}\right) + f\left(\frac{\alpha q^{j}-b'}{2a'}-\frac{1}{2}\right) + f\left(\frac{\beta q^{j}-b'}{2a'}+\frac{1}{2}\right)\\
&+ f\left(\frac{\beta q^{j}-b'}{2a'}-\frac{1}{2}\right) + f\left(\frac{\gamma q^{j}-b'}{2a'}+\frac{1}{2}\right) + f\left(\frac{\gamma q^{j}-b'}{2a'}-\frac{1}{2}\right)\Bigg)
\end{align*}
is $4$ times an integer, so
\[
\alpha\beta\gamma d \equiv 3\frac{gq^j}{2} \pmod{4},
\]
while if $\frac{b''-b'}{a'}$ is even then $d \equiv 0 \pmod{4}$.

From here the argument is very similar to the proof of Theorem \ref{oddcubic}, with $g$ in the place of $a$ and $q$ in the place of $p$. The only part of the argument which does not directly generalize is showing that we may choose $u$ such that $\alpha\beta\gamma\frac{d-4f(u)}{\frac{1}{2}gq^j} \equiv 3 \pmod{8}$ if $\frac{b''-b'}{a'}$ is odd, or such that $\alpha\beta\gamma\frac{d-4f(u)}{\frac{1}{2}gq^j}$ is $4$ times something which can be primitively represented as a sum of three squares if $\frac{b''-b'}{a'}$ is even.

Suppose first that $f:\mathbb{Z}_2\rightarrow\mathbb{Z}_2$ is surjective. Then since $\alpha\beta\gamma(d-4f(u))$ automatically has the right congruence class modulo $4$, we may pick $u$ modulo the largest power of $2$ dividing $4g$ to make $\alpha\beta\gamma\frac{d-4f(u)}{\frac{1}{2}gq^j}$ either $3$ mod $8$ or $8$ mod $16$. Otherwise, we have $\gcd(a,b) \equiv 2 \pmod{4}$, so $v_2(g) \in \{1,2\}$, and $q$ is odd.

Now suppose $\frac{b''-b'}{a'}$ is odd. By the choice of $b''$, in this case we have $g\equiv 2\pmod{4}$ (if $g \equiv 4 \pmod{8}$ then $b'$ is odd). Then $\alpha\beta\gamma\frac{d}{\frac{1}{2}gq^j}$ is automatically $3$ modulo $4$, and since $f(u)$ takes both even and odd values we can ensure that $\alpha\beta\gamma\frac{d-4f(u)}{\frac{1}{2}gq^j} \equiv 3 \pmod{8}$.

Finally suppose $\frac{b''-b'}{a'}$ is even. Then $b'$ must be odd, so $g\equiv 4\pmod{8}$, $b \equiv 2\pmod{4}$ and $4\mid a$, so that the value of $f(x)$ mod $4$ is determined by $x$ mod $2$. In this case $d$ is automatically a multiple of $4$. By choosing $u$ to be congruent to $n$ mod $2$, we can ensure that $d-4f(u)$ is a multiple of $8$, and then by the choice of $b''$ we will have $d-4f(u) \equiv 8 \pmod{16}$. Thus $\alpha\beta\gamma\frac{d-4f(u)}{2gq^j}$ is odd, and we just need to choose $u$, congruent to $n$ mod $2$, such that $\alpha\beta\gamma\frac{d-4f(u)}{2gq^j}$ isn't $7$ mod $8$.

In order to finish we just need to show that if $4\mid a$, $b \equiv 2 \pmod{4}$, and $c$ is odd, then as $u$ varies over numbers congruent to $n$ mod $2$, $f(u)$ takes at least two distinct values mod $16$. If $n$ is even, then from $f(2) = a+b+2c, f(4) = 10a+6b+4c$ we see that $f(4) \equiv 2f(2)+8 \pmod{16}$, so that at least one of $f(2), f(4)$ is not zero mod $16$. If $n$ is odd, then from $f(1) = c, f(-1) = b-c$ we see that if $b \not\equiv 2c \pmod{16}$ then $f(1), f(-1)$ are different mod $16$, and otherwise from $f(3) = 4a+3b+3c$ we see $f(3) \equiv 3b+3c \equiv 9c \not\equiv c \equiv f(1) \pmod{16}$.
\end{proof}

\section{An explicit result for sums of seven octahedrals}

We will need one special case of Lemma \ref{linnik}.

\begin{repthm}{quadratic} If $m, m' > e^{6.6\cdot 10^6}$, $m, m' \equiv \pm 2 \pmod{5}$, $m\equiv 3\pmod{4}$, and $m' \equiv 2\pmod{4}$, then at least one of $m, m'$ can be primitively represented by the quadratic form $83x^2 + 91y^2 + 99z^2$.
\end{repthm}

We will give the proof of this theorem in the next section.

\begin{rem} The number $689469562$ is not represented by $83x^2 + 91y^2 + 99z^2$. On the other hand, a computer search shows that every number $m$ strictly between $689469562$ and $10^{10}$ which is congruent to $2$ or $3$ modulo $4$ and congruent to $2$ or $3$ modulo $5$ is primitively represented by the form $83x^2 + 91y^2 + 99z^2$.
\end{rem}

\begin{thm}\label{octahedral}
Every number $n$ greater than $e^{10^7}$ is a sum of seven positive octahedral numbers.
\end{thm}
\begin{proof}
To save writing, we set $\alpha(x) = \frac{2x^3+x}{3}$. Note that for any $t$ we have
\[
n - 2\alpha(83t) - 2\alpha(91t) - 2\alpha(99t) = n - 3060876t^3 - 182t \equiv n-t^3-2t \mbox{ (mod }5).
\]
First find a congruence class $t \equiv \pm 1 \pmod{5}$ such that one of $n-t^3+t, n-t^3$ is congruent to one of $1,-1$ modulo $5$.

Now we claim that we may find a number $T$ of the form $2^a3^b$ which is congruent to $t$ modulo $5$ and satisfies the inequalities
\[
\alpha(144T) + 4e^{6.6\cdot 10^6}T < n-2\alpha(83T)-2\alpha(91T)-2\alpha(99T) < 4\cdot 83^3T^3.
\]
To do so, we note that these inequalities are implied by the inequalities
\[
\frac{\sqrt[3]{n}}{172} > T > \frac{\sqrt[3]{n}}{174}:
\]
if $T > \frac{\sqrt[3]{n}}{174}$ then we have
\[
(4\cdot 83^3 + 3060876)T^3 + 182T > 174^3T^3 > n,
\]
and from $T^2 > e^{\frac{2}{3}10^7-2\log(174)} > e^{6.6\cdot 10^6}$ and $\frac{\sqrt[3]{n}}{172} > T$ we have
\begin{align*}
\left(\frac{2}{3}144^3 + 3060876\right)T^3 + \left(4e^{6.6\cdot 10^6}+\frac{1}{3}144 + 182\right)T\\
 < \left(\frac{2}{3}144^3 + 3060876 + 5\right)T^3 < 172^3T^3 < n.
\end{align*}
Now we note that $2^{-336}3^{212} \equiv 1 \pmod{5}$ and $\frac{174}{172} > 2^{-336}3^{212} > 1.008$. To find a $T$ between $\frac{\sqrt[3]{n}}{174}$ and $\frac{\sqrt[3]{n}}{172}$ we start by choosing $a\in \{0,2\}$ such that $2^{a} \equiv t\pmod{5}$. Next we choose $k$ as large as possible such that $2^{4k+a}$ is smaller than $\frac{\sqrt[3]{n}}{172}$, so in particular $2^{4k+a} > \frac{\sqrt[3]{n}}{16\cdot 172}$ and $k > 10^6$. Next we pick $j$ such that
\[
T = 2^{4k+a-336j}3^{212j}
\]
is between $\frac{\sqrt[3]{n}}{174}$ and $\frac{\sqrt[3]{n}}{172}$. We will then necessarily have $j < \frac{\log(16)}{\log(1.008)} < 348$, so $4k+a-336j > 0$ and $T$ is an integer, congruent to $t$ modulo $5$.

Next find a number $u$ between $0$ and $144T$ satisfying the congruences
\[
\alpha(u) \in \{n-t^3+t, n-t^3\} \cap \{1,-1\} \pmod{5}
\]
and
\[
\alpha(u) \equiv n-2\alpha(83T)-2\alpha(91T)-2\alpha(99T)-12T \pmod{16T}.
\]
To see that this is possible, first note that for a fixed nonzero value of $\alpha(u)$ modulo $5$, there are two choices for $u$ modulo $5$, with difference congruent to $\pm 1$ modulo $5$: we have $\alpha(1)\equiv \alpha(2) \equiv 1\pmod{5}$ and $\alpha(3)\equiv \alpha(4)\equiv -1\pmod{5}$. 

Thus for a fixed value of $u$ modulo $48T$, determined by the congruence modulo $16T$ (using Corollary \ref{octa}), there are two possible solutions for $u$ modulo $240T$, with difference congruent to $\pm 96T$ modulo $240T$ (since $96T \equiv T \equiv t \equiv \pm 1\pmod{5}$), and so at least one of them is bounded by $240T-96T = 144T$. Similarly find $u'$ between $0$ and $144T$ such that $\alpha(u') \equiv \alpha(u) \pmod{5}$ and $\alpha(u') \equiv \alpha(u) + 4T \pmod{16T}$.

If we let
\begin{align*}
m &= \frac{n-2\alpha(83T)-2\alpha(91T)-2\alpha(99T)-\alpha(u)}{4T},\\
m' &= \frac{n-2\alpha(83T)-2\alpha(91T)-2\alpha(99T)-\alpha(u')}{4T},
\end{align*}
then we have $m\equiv \pm 2 \pmod{5}$, $m\equiv 3\pmod{4}$, $e^{6.6\cdot 10^6} < m < 83^3T^2$, and similarly for $m'$ except $m' \equiv 2 \pmod{4}$. By Theorem \ref{quadratic}, one of $m, m'$ is primitively represented by $83x^2+91y^2+99z^2$. Assume that $m$ can be primitively represented (the other case is similar). By the bounds on $m$, we have $x,y,z < 83T$. Thus we can write
\begin{align*}
n =& \; 4T(83x^2+91y^2+99z^2) + 2\alpha(83T)+2\alpha(91T)+2\alpha(99T)+\alpha(u)\\
=& \; \alpha(83T+x)+\alpha(83T-x) + \alpha(91T+y)+\alpha(91T-y)+\alpha(99T+z)+\alpha(99T-z)+\alpha(u).\qedhere
\end{align*}
\end{proof}

\section{Proof of Theorem \ref{quadratic}}

For the most part, this section follows Linnik's and Maly\v{s}ev's arguments from \cite{linnik-quaternions}. All of the constants in their argument are made explicit, and several of their bounds are sharpened as a result. Additionally, a minor correction is made to Maly\v{s}ev's correction of Note 13 (see Lemma \ref{note13}).

We will use the following notation: $\omega(n)$ is the number of distinct prime factors of $n$, $\Omega(n)$ is the number of prime factors of $n$ counted with multiplicity, $\tau(n)$ is the number of positive divisors of $n$, and $n^{[\frac{1}{2}]}$ is the largest integer whose square divides $n$. Also, let $t(n)$ be the number of primitive representations of $n$ as a sum of three integer squares.

\begin{defn} $\mathbb{B} = \mathbb{Z}[i,j,k,\frac{1+i+j+k}{2}]$ is the set of \emph{Hurwitz quaternions}. A quaternion $X$ is called a \emph{vector} if it has trace $0$. It is called \emph{proper} if it can't be written as an integer greater than $1$ times another Hurwitz quaternion. The \emph{norm} of $X$ is defined to be $\mbox{Nm}(X) = X\bar{X}$, the \emph{trace} of $X$ is defined to be $\mbox{Tr}(X) = X + \bar{X}$, and the \emph{real part} of $X$ is defined to be $\mbox{Re}(X) = \frac{1}{2}\mbox{Tr}(X)$.
\end{defn}

We will frequently use the fact that the ring of Hurwitz quaternions is a left (respectively right) Euclidean domain under the usual norm, as well as other basic facts about the arithmetic of the quaternions. These facts are reviewed in Linnik and Maly\v{s}ev's article \cite{linnik-quaternions}.

\begin{defn}
Let $r = 83\cdot 91\cdot 99$. An integral quaternion $R$ with norm $r$ is called \emph{good} if the set of integral quaternions orthogonal to $R$ has a basis $A,B,C$ such that $\mbox{Nm}(A) = 83, \mbox{Nm}(B) = 91, \mbox{Nm}(C) = 99$, and $A, B, C$ are mutually orthogonal.
\end{defn}

\begin{rem} There is a correspondence between good quaternions $R$ of norm $r$ and representations of the ternary quadratic form $83x^2+91y^2+99z^2$ as a sum of four squares of linear forms. If $A = a_0 + a_1i + a_2j + a_3k$, and similarly for $B, C$, then
\[
83x^2 + 91y^2 + 99z^2 = \sum_{i=0}^3 (a_ix + b_iy + c_iz)^2.
\]
Thus we can use Gauss's theory of quadratic forms to systematically find good quaternions.
\end{rem}

\begin{lem}\label{R} A positive number $m$ can be represented by the form $83x^2 + 91y^2 + 99z^2$ if and only if we can simultaneously solve the quaternion equations
\[
L = RX,\ L^2 = -rm.
\]
for $L$ and $X$ Hurwitz quaternions, and $R$ a good quaternion of norm $r$. If $L$ is proper, then $m$ can be represented primitively.
\end{lem}
\begin{proof} If $L,X$ are as above, then we have $\mbox{Re}(L) = 0$, so $\bar{X}$ is orthogonal to $R$. Let $A,B,C$ be as in the above definition. Then we can write $\bar{X} = xA + yB + zC$ for integers $x,y,z$. Since $\mbox{Nm}(X) = \frac{\mbox{Nm}(L)}{\mbox{Nm}(R)} = m$, we have
\[
m = \mbox{Nm}(xA + yB + zC) = 83x^2 + 91y^2 + 99z^2.
\]
Conversely, for any representation $m = 83x^2 + 91y^2 + 99z^2$ we can take $\bar{X} = xA + yB + zC$, $L = RX$, and we will have $\mbox{Re}(L) = 0$, so $L^2 = -\mbox{Nm}(L) = -rm$.
\end{proof}

\begin{defn} A proper quaternion $Q$ with norm a power of $5$ is called a \emph{tourist quaternion} if for every proper quaternion $R$ of norm $r$ we can write $Q = Q_1Q_2$ and $Q_2R = R'Q_3$, with $Q_1,Q_2,Q_3$ Hurwitz quaternions and $R'$ a good quaternion of norm $r$.
\end{defn}

\begin{defn} A quaternion $Q$ with norm congruent to $1$ modulo $4$ is said to be in \emph{standard form} if it is congruent to $1$ modulo $2\mathbb{B}$ and $\mbox{Tr}(Q) \equiv 2 \pmod{8}$.
\end{defn}

\begin{rem} Any quaternion with norm congruent to $1$ modulo $4$ has a unique left (respectively right) associate which is in standard form.
\end{rem}

\begin{lem}\label{Q} If there is an integral quaternion $L$ with $L^2 = -rm$, an integer $l$, and a tourist quaternion $Q$ such that
\[
l + L = UQ
\]
for some integral quaternion $U$, then $m$ can be represented by the quadratic form $83x^2 + 91y^2 + 99z^2$. If $L$ is proper and $m$ is relatively prime to $5$, then $m$ can be represented primitively.
\end{lem}
\begin{proof} Let $q = \mbox{Nm}(Q)$. Since $q$ and $r$ are relatively prime, we can find $u,v$ such that $qu + rv = l$. Then we have
\[
rv + L = (U-u\bar{Q})Q.
\]
Taking norms of both sides, we see that $\mbox{Nm}(U-u\bar{Q})$ is a multiple of $r$, so we can write $U-u\bar{Q} = R'V$ for some proper quaternion $R'$ with norm $r$. Write $Q = Q_1Q_2$ such that $Q_2R' = RQ_3$ with $Q_1,Q_2,Q_3$ Hurwitz quaternions and $R$ a good quaternion. Then
\[
rv + Q_2LQ_2^{-1} = Q_2R'VQ_1 = RQ_3VQ_1.
\]
Write $L' = Q_2LQ_2^{-1}$. Since the right hand side of the above is an integral quaternion, $L'$ is integral as well. Thus
\[
L' = R(Q_3VQ_1 - v\bar{R}),\ L'^{2} = -rm,
\]
so by Lemma \ref{R} $m$ can be represented by the form $83x^2 + 91y^2 + 99z^2$.
\end{proof}

Next we check that most proper quaternions of norm $5^s$ are in fact tourist quaternions.

\begin{lem}\label{walk} The number of proper quaternions $Q$ with norm $5^s$ which are in standard form and are not tourist quaternions is at most
\[
1625472\cdot 5^{s-\lfloor\frac{s}{11422}\rfloor} \le 8127360\cdot 5^{\frac{11421}{11422}s}.
\]
\end{lem}
\begin{proof} Every proper quaternion of norm $5^s$ which is in standard form can be written in a unique way as a product of the quaternions $1\pm 2i, 1\pm 2j, 1\pm 2k$ such that no term is immediately followed by its conjugate (this is a consequence of the version of the unique factorization theorem which holds for Hurwitz quaternions, but it can also be checked directly). In other words, there is a one-to-one correspondence with reduced words of length $s$ in the free group on three elements. The total number of such $Q$ is $\frac{6}{5}5^s$, which can be checked by induction on $s$.

Let $G$ be the following multi-graph: the vertices of $G$ correspond to the proper quaternions of norm $r$ up to right-multiplication by units, and there is an edge between vertices corresponding to quaternions $R$ and $R'$ for each one of $1\pm 2i, 1\pm 2j, 1\pm 2k$ in the set $\{\lambda\in \mathbb{B}\mid \lambda R = R'\lambda\}$. For a fixed $R$, and a quaternion $Q$ as above, if we write
\[
Q = Q_s\cdots Q_1,
\]
each $Q_i$ one of $1\pm 2i, 1\pm 2j, 1\pm 2k$, then by letting each suffix $Q_i\cdots Q_1$ of $Q$ act on $R$ we obtain a non-backtracking walk on $G$ starting at $R$ (a \emph{non-backtracking} walk is a walk which never immediately travels back along an edge it just traversed). $Q$ is a tourist quaternion if and only if each such walk, starting at each possible $R$, passes through a good vertex at some point. Thus it is enough to find an upper bound on the probability that a random non-backtracking walk of length $s$ starting at any vertex of $G$ fails to hit a good vertex.

The multi-graph $G$ turns out to have several nice properties. First, it is connected - this is checked in general by Linnik and Maly{\v{s}}ev \cite{linnik-quaternions} using facts about representations of large numbers by quadratic forms in four variables. Second, it has no multiple edges or loops (so it is in fact a graph). Third, the graph $G$ is even a Ramanujan graph by a result of Lubotzky, Phillips, and Sarnak \cite{ramanujan-graph} (their proof assumes $r$ is prime for convenience of exposition, but this assumption can be easily removed) - we will not use this result, but it is useful in the general case if $r$ is very large and the graph $G$ becomes too large to compute (see \cite{akshay} for details).

The four vertices $R_1 = 216 + 365i + 421j + 625k, R_2 = 216 + 409i + 443j + 581k, R_3 = 736 + 99i + 155j + 415k, R_4 = 404 + 99i + 487j + 581k$ are all good: the orthogonal complement to $R_1$ is spanned by
\[
A_1 = 3 - 3i + 7j - 4k, B_1 = 1 - 8i - j + 5k, C_1 = 9 + i - 4j - k,
\]
the orthogonal complement of $R_2$ is spanned by
\[
A_2 = 3 - 7i + 5j, B_2 = 1 + 4i + 5j - 7k, C_2 = 9 + i - 4j - k,
\]
the orthogonal complement of $R_3$ is spanned by
\[
A_3 = 1 - 9i + j, B_3 = 1 + i + 8j - 5k, C_3 = 5 - 5j - 7k,
\]
and the orthogonal complement of $R_4$ is spanned by
\[
A_4 = 1 - 9i + j, B_4 = 5 + i + 4j - 7k, C_4 = 7 - 7j + k.
\]
If we consider all possible ways of permuting or changing the signs of the coefficients of $R_1, R_2, R_3, R_4$, we find a total of $4\cdot\frac{2^4\cdot 4!}{8} = 192$ good vertices in the graph $G$.

The number of proper quaternions of a given odd norm $r$ up to right multiplication by units is given by the formula
\[
r\prod_{p\mid r} \frac{p+1}{p},
\]
which follows from the version of the unique factorization theorem which holds for Hurwitz quaternions.
Thus the graph $G$ only has $(83+1)(7+1)(13+1)(3+1)3(11+1) = 1354752$ vertices (this can also be checked by building the graph and counting the number of vertices), so it is not too hard to write a computer program to compute $G$ and then to count the number of non-backtracking walks of length $11422$ which never pass through one of the $192$ good vertices and do not start by following a specified edge. The details of the computation are as follows. First, pick a known vertex in $G$. From any given vertex of $G$ one can easily find its neighbours by acting on the corresponding quaternion with each of $1\pm 2i, 1\pm 2j, 1\pm 2k$, and so one can build the whole graph $G$ using a breadth first search, keeping track of the correspondence between vertices and quaternions using a hash table and storing $G$ as an adjacency list rather than an adjacency matrix to take advantage of the fact that every vertex has degree $6$. Then iteratively find the number of non-backtracking walks of every given length $l$ from every vertex which avoid each given starting edge on that vertex and never pass through a good vertex - note that in order to compute the number of such walks of length $l$ we only need to know the number of such walks of length $l-1$ at the adjacent vertices. Since the number of walks grows exponentially, after each step with $l > 24$ the number of such walks is divided by $5$ and rounded up in order to avoid integer overflow (this makes almost no difference to the computation, and still provides an upper bound on the number of such walks). The code used to perform this computation is located on the arxiv together with this preprint, and the whole computation takes about two hours on a laptop.

The number of such walks of length $11422$ is then found to be strictly less than $5^{11421}$, regardless of the starting vertex (and the specified edge). Thus, the probability that a random non-backtracking walk of length $s$ avoids the $192$ good vertices is at most $5^{-\lfloor\frac{s}{11422}\rfloor}$, regardless of the starting point.

Since the number of quaternions $Q$ which are not tourist quaternions is at most the number of bad vertices times the maximum number of non-backtracking walks of length $s$ which never pass through a good vertex starting from any one of them, we see that there are at most
\[
(1354752-192)\cdot \frac{6}{5}5^{s-\lfloor\frac{s}{11422}\rfloor} = 1625472\cdot 5^{s-\lfloor\frac{s}{11422}\rfloor}
\]
such quaternions $Q$.
\end{proof}

We will defer the more technical lemmas needed for the proof of Lemma \ref{distinct} and Theorem \ref{quadratic} in order to maintain the flow of the argument.

\begin{lem}\label{distinct} Let $0 < \tau < \frac{1}{2}$ be a fixed real number. Let $s$ be an integer such that $(rm)^{\frac{1}{2}+\tau} \le 5^s < 5(rm)^{\frac{1}{2}+\tau}$. Let $l$ be an integer such that $5^s\mid\mid l^2+rm$. Then the number of distinct proper integral quaternions $Q$ of norm $5^s$ in standard form and such that we can find $L, U$ satisfying
\[
l + L = QU
\]
where $L$ is a proper integral vector of norm $rm$ and $U$ is an integral quaternion, is at least
\[
\frac{t(rm)^2}{2t(rm) + 96S},
\]
where $S$ is the sum
\[
S = \sum_{c=1}^{6(rm)^{\frac{1}{2}-\tau}}\sum_{5^{s_1}\mid c}\sum_{d=1}^{5^{s_1}\sqrt{\frac{4rm}{5^sc}}}2^{\omega(5c)}\tau(4rm-5^{s-2s_1}d^2c)\left(4rm, c\right)^{[\frac{1}{2}]},
\]
if $(rm)^{\tau} > \frac{8\sqrt{3}}{3}$ and $5\nmid rm$. In particular, if $t(rm) \ge 96S$ then the number of such $Q$ is at least $\frac{t(rm)}{3}$.
\end{lem}
\begin{proof} First we note that there is a bijection between pairs of $L, U$ as above and pairs of $L', U'$ as above but with $Q$ replaced with $\bar{Q}$. This bijection is given by taking $U' = \bar{U}$, $L' = \bar{Q}\bar{U}-l = \bar{Q}\bar{L}\bar{Q}^{-1}$. Since $5\nmid \mbox{Nm}(L)$, $L'$ will be proper if $L$ is, and conversely.

Thus the number of pairs of proper vectors $L_1, L_2$ of norm $m$ that correspond to a common $Q$ is the same as the number of pairs $L_1, L_2$ satisfying
\begin{align*}
l+L_1 &= QU_1,\\
l+L_2 &= \bar{Q}U_2.
\end{align*}
By Lemmas \ref{junior} and \ref{senior}, the number of such pairs is then at most the total number of possible quaternions $L_1$ plus $t(rm) + 96S$.

Now we apply the Cauchy-Schwarz inequality to see that the total number of pairs $(L_1, L_2)$ as above times the number of distinct quaternions $Q$ is at least $t(rm)^2$.
\end{proof}

\begin{thm}\label{quadratic} If $m, m' > e^{6.6\cdot 10^6}$, $m, m' \equiv \pm 2 \pmod{5}$, $m\equiv 3\pmod{4}$, and $m' \equiv 2\pmod{4}$, then at least one of $m, m'$ can be primitively represented by the quadratic form $83x^2 + 91y^2 + 99z^2$.
\end{thm}
\begin{proof} Note first that $rm$ and $rm'$ must have different squarefree parts. By Lemma \ref{siegel} applied to $rm, rm'$ with $\epsilon = 10^{-6}$ we may suppose without loss of generality that
\begin{align*}
t(rm) &\ge \frac{12(rm)^{\frac{1}{2}-10^{-6}}}{10^6\pi}\prod_{p\text{ \rm odd prime}} \min\left(1,p^{2\cdot 10^{-6}}\left(1-\frac{1}{p}\right)\right)\\
&> \frac{12(rm)^{\frac{1}{2}-10^{-6}}}{9\cdot 10^6\pi}.
\end{align*}
Let $\tau = \frac{1}{26510}$. Note that $\frac{11421}{11422}(\frac{1}{2} + \tau) < \frac{1}{2} - \frac{1}{165200}$. Let $s$ be chosen such that $(rm)^{\frac{1}{2}+\tau} \le 5^s < 5\cdot (rm)^{\frac{1}{2}+\tau}$. Then for $rm > e^{6.6\cdot 10^6}$ we have
\[
8127360\cdot 5^{\frac{11421}{11422}s} < 5\cdot 8127360\cdot (rm)^{\frac{1}{2}-\frac{1}{165200}} < \frac{t(rm)}{3}.
\]
Thus by Lemma \ref{walk} any set of $\frac{t(rm)}{3}$ proper integral quaternions of norm $5^s$ in standard form contains at least one tourist quaternion.

We need to bound the sum
\[
S = \sum_{c=1}^{6(rm)^{\frac{1}{2}-\tau}}\sum_{5^{s_1}\mid c}\sum_{d=1}^{5^{s_1}\sqrt{\frac{4rm}{5^sc}}}2^{\omega(5c)}\tau(4rm-5^{s-2s_1}d^2c)\left(4rm, c\right)^{[\frac{1}{2}]}
\]
by $\frac{t(rm)}{96}$ in order to apply Lemma \ref{distinct}. By Lemma \ref{sum}, for $rm > e^{6.6\cdot 10^6}$ we have
\begin{align*}
S \le &\ e^{75}(rm)^{10^{-6}}\log(rm)^9(rm)^{\frac{1}{2}-\tau}\\
\le &\ e^{75}(rm)^{\frac{1}{2}+\frac{1}{46635}+10^{-6}-\tau}\\
\le &\ \frac{t(rm)}{96}.
\end{align*}

Now choose $l$ such that $5^s\mid\mid l^2 + rm$. We can do this since $\left(\frac{-rm}{5}\right) = \left(\frac{\pm 1}{5}\right) = 1$. Let $L_1, ..., L_{t(rm)}$ be the set of all proper vectors of norm $rm$. For each $1 \le a \le t(rm)$, we have $5^s \mid\mid \mbox{Nm}(l+L_a) = l^2 + rm$, so to each one there corresponds an equation
\[
l+L_a = Q_aU_a,
\]
where $Q_a$ is a proper quaternion of norm $5^s$ in standard form. By Lemma \ref{distinct}, there are at least $\frac{t(rm)}{3}$ distinct quaternions $Q_a$. Thus for some $a$, $\bar{Q}_a$ is a tourist quaternion, and we have
\[
l + \bar{L}_a = \bar{U}_a\bar{Q}_a,
\]
so by Lemma \ref{Q} $m$ is primitively represented by $83x^2 + 91y^2 + 99z^2$.
\end{proof}

In what follows we shall absorb the constant $r$ into the value of $m$. First we need to prove an effective form of Siegel's theorem which is applicable to numbers which may not be squarefree.

\begin{lem}\label{siegel} If $m, m'$ are positive integers with different squarefree parts which are not multiples of $4$ and not congruent to $7\pmod{8}$, and if $0 \le \epsilon < 10^{-3}$, then we have
\[
\max\left(\frac{t(m)}{m^{\frac{1}{2}-\epsilon}},\ \frac{t(m')}{m'^{\frac{1}{2}-\epsilon}}\right) \ge \frac{12\epsilon}{\pi}\prod_{p\text{ \rm odd prime}} \min\left(1,p^{2\epsilon}\left(1-\frac{1}{p}\right)\right).
\]
\end{lem}
\begin{proof} It's well known that for any positive $m$ which is not a multiple of $4$ and not congruent to $7\pmod{8}$, we have
\[
t(m) = \frac{24\sqrt{m}}{\pi}\begin{cases}L(1,\chi_{-4m}) & \mbox{if } m\equiv 1,2\pmod{4}\\L(1,\chi_{-m}) & \mbox{if } m\equiv 3\pmod{8}\end{cases},
\]
where $\chi_{-4m}(n) = \left(\frac{-4m}{n}\right), \chi_{-m}(n) = \left(\frac{-m}{n}\right)$. For simplicity we restrict to the case $m, m' \equiv 1,2 \pmod{4}$ for the remainder of the proof (the proof is almost unchanged in the case that one or both is congruent to $3$ modulo $8$).

Write $m = dk^2, m' = d'k'^2$, with $d, d'$ squarefree. Then by assumption we have $d \ne d'$. From Lemma 2 of \cite{jilu}, for any $1-\epsilon < s < 1$ such that $L(s,\chi_{-4d}) \ge 0$ and for any $x \ge 1$ we have
\[
L(1,\chi_{-4d}) \ge (1-s)\left(\frac{1}{x^{1-s}}\left(\frac{\pi^2}{6}-\frac{4}{\lfloor\sqrt{x}\rfloor}\right)-\frac{8d}{x^{\frac{3}{2}}}\frac{\zeta(3/2)^2 3!}{5\pi^2}\right).
\]
Plugging in $x = 10^4d$, and using $(10^4)^{1-s} \le 10^{4\epsilon} < 1.01$, we see that
\[
L(1,\chi_{-4d}) \ge \frac{3}{2}\frac{1-s}{d^{1-s}}.
\]
Taking $s = 1-\frac{\epsilon}{3}$, we see that either we have
\[
L(1,\chi_{-4d}) \ge \frac{\epsilon}{2d^{\frac{\epsilon}{3}}} \ge \frac{\epsilon}{2d^{\epsilon}},
\]
or else $L(1-\frac{\epsilon}{3},\chi_{-4d}) < 0$, in which case there is some $\beta$ with $1-\frac{\epsilon}{3} < \beta < 1$ such that $L(\beta,\chi_{-4d}) = 0$.

Assume for contradiction that $L(1,\chi_{-4d}) < \frac{\epsilon}{2d^{\epsilon}}$ and $L(1,\chi_{-4d'}) < \frac{\epsilon}{2d'^{\epsilon}}$. Note that this implies $d, d' \ge \frac{4}{\epsilon^2} > 10^6$. By the above we can find $1-\frac{\epsilon}{3} < \beta, \beta' < 1$ with $L(\beta,\chi_{-4d}) = L(\beta',\chi_{-4d'}) = 0$. From \cite{kadiri}, we have
\[
\min(\beta, \beta') \le 1 - \frac{1}{K\log(16dd')}
\]
with $K = 2.0452$. Let $\mathbb{Q}(\sqrt{D})$ be the real quadratic subfield of $\mathbb{Q}(\sqrt{-d},\sqrt{-d'})$, $D$ a fundamental discriminant dividing $4dd'$. Take $s = \min(\beta,\beta')$, so that
\[
1-\frac{\epsilon}{3} < s < 1 - \frac{1}{K\log(16dd')}.
\]
By Lemma 2 of \cite{jilu} again, for $x \ge 1$ we have
\[
L(1,\chi_{-4d})L(1,\chi_{-4d'})L(1,\chi_D) \ge (1-s)\left(\frac{1}{x^{1-s}}\left(\frac{\pi^2}{6}-\frac{6}{\lfloor\sqrt{x}\rfloor}\right)-\frac{32d^2d'^2}{x^{\frac{3}{2}}}\frac{\zeta(3/2)^4 5!}{13\pi^4}\right).
\]
Plugging in $x = 10^4(dd')^{\frac{3}{2}}$, and using the bound $L(1,\chi_D) \le \frac{\log(4dd')+1.44}{2}$ from \cite{jilu}, we get
\[
\frac{\epsilon^2}{4(dd')^{\epsilon}} > L(1,\chi_{-4d})L(1,\chi_{-4d'}) \ge \frac{3(1-s)}{(dd')^{\frac{3}{2}(1-s)}(\log(4dd')+1.44)}.
\]
We claim that this is a contradiction. Since the right hand side is a unimodal function of $s$, it suffices to check that this is impossible when $s$ is either $1-\frac{\epsilon}{3}$ or $1-\frac{1}{K\log(16dd')}$. When $s = 1-\frac{\epsilon}{3}$, we get
\[
\frac{\epsilon}{4(dd')^{\frac{\epsilon}{2}}} > \frac{1}{\log(4dd')+1.44},
\]
but the left hand side is easily seen to be at most $\frac{1}{2e\log(dd')}$, so this case is impossible for $dd' \ge 10^{12}$. When $s = 1-\frac{1}{K\log(16dd')}$, we get
\[
\left(\frac{\epsilon}{2(dd')^{\frac{\epsilon}{2}}}\right)^2 > \frac{3}{e^{\frac{3}{2K}}K\log(16dd')(\log(4dd')+1.44)}.
\]
Since the left hand side is at most $\left(\frac{1}{e\log(dd')}\right)^2$, we easily see that this is impossible for $dd' \ge 10^{12}$. This completes the contradiction.

Thus, we may assume without loss of generality that $L(1,\chi_{-4d}) \ge \frac{\epsilon}{2d^{\epsilon}}$. From
\[
L(s,\chi_{-4m}) = L(s,\chi_{-4d})\prod_{p\mid k,\ p\text{ prime}} \left(1-\frac{\chi_{-4d}(p)}{p^s}\right),
\]
we have
\[
\frac{L(1,\chi_{-4m})}{L(1,\chi_{-4d})} \ge \prod_{p\mid k,\ p\text{ prime}} \left(1-\frac{1}{p}\right) \ge \frac{1}{k^{2\epsilon}}\prod_{p\text{ odd prime}} \min\left(1,p^{2\epsilon}\left(1-\frac{1}{p}\right)\right)
\]
Putting these together we get the desired inequality.
\end{proof}

Finally, we need to prove an upper bound on the number of \emph{conjugate pairs}, i.e. pairs $(L_1, L_2)$ such that
\[
l+L_1 = QU_1,\ l+L_2 = \bar{Q}U_2,
\]
where $Q$ is a quaternion of norm $5^s$.
Linnik and Maly\v{s}ev \cite{linnik-quaternions} cleverly transform this problem into a question of bounding the number of representations of a binary quadratic form as a sum of three squares of linear forms. The situation is somewhat complicated by the fact that the representations we need to consider may not be ``proper'', and furthermore that the discriminants involved need not be squarefree.

\begin{lem}\label{square-roots} The number of solutions $x$ to the congruence
\[
x^2 \equiv a \pmod{b}
\]
is at most $2^{1+\omega\left(\frac{b}{(a,b)}\right)}(a,b)^{[\frac{1}{2}]}$.
\end{lem}
\begin{proof} First consider the case $b = p^k$, for $p$ a prime. Let $p^l \mid\mid a$. If $l \ge k$, then any solutions $x$ are multiples of $p^{\lceil \frac{k}{2}\rceil}$, so the number of solutions $x$ is at most $p^{\lfloor \frac{k}{2}\rfloor} = b^{[\frac{1}{2}]} = (a,b)^{[\frac{1}{2}]}$. So we may assume without loss of generality that $l < k$ and $l$ is even, say $l = 2m$.

Now we have $x = p^my$, $y^2 \equiv \frac{a}{p^l} \pmod{p^{k-l}}$. Each solution $y$ modulo $p^{k-l}$ gives rise to $p^m = (a,b)^{[\frac{1}{2}]}$ distinct solutions $x$ modulo $p^k$. If $p$ is odd, there are at most two such $y$ modulo $p^{k-l}$, and if $p$ is $2$ there are at most four such solutions $y$ modulo $2^{k-l}$.

Putting the above together using the Chinese Remainder Theorem, we obtain the Lemma.
\end{proof}

\begin{lem}\label{gauss-ternary} Let $\phi = px^2+2qxy+ry^2$ be a quadratic form with determinant $d = pr-q^2$. A representation of $\phi$ as a sum of three squares of linear forms
\[
\phi(x,y) = (a_1x+b_1y)^2 + (a_2x+b_2y)^2 + (a_3x+b_3y)^2
\]
is called \emph{proper} if
\[
(a_2b_3-a_3b_2, a_3b_1-a_1b_3, a_1b_2-a_2b_1) = 1.
\]
The number of proper representations of $\phi$ as a sum of three squares is $0$ if $(p,q,r) \ne 1$, and otherwise it is at most
\[
48\cdot 2^{\omega(d)}.
\]
\end{lem}
\begin{proof} We will follow the argument from Venkov \cite{venkov}, Chapter $4$, section $14$. If we transform the quadratic form $x^2+y^2+z^2$ by a matrix
\[
\left(\begin{array}{ccc} a_1 & b_1 & c_1 \\ a_2 & b_2 & c_2 \\ a_3 & b_3 & c_3 \end{array}\right),
\]
with determinant $1$ then we get a positive definite ternary quadratic form $g$ of determinant $1$ with matrix
\[
\left(\begin{array}{ccc} p & q & m \\ q & r & n \\ m & n & s \end{array}\right).
\]
Let the adjoint $G$ of $g$ have the matrix
\[
\left(\begin{array}{ccc} P & Q & M \\ Q & R & N \\ M & N & d \end{array}\right).
\]
Venkov checks that if we vary the choice of $c_1,c_2,c_3$, the values $M,N$ vary by multiples of $d$ \cite{venkov}. Furthermore, they satisfy
\[
Pd - M^2 = r,\ Rd - N^2 = p,\ Qd-MN = -q.
\]
Taking these equations mod $d$, we find $M^2 \equiv -r, MN \equiv q, N^2 \equiv -p \pmod{d}$. Now if $(p,q,r) \ne 1$ we find that $(M,N,d) \ne 1$, contradicting the fact that $G$ has determinant $1$. Otherwise, we easily see that there are at most $2^{1+\omega(d)}$ ordered pairs $(M,N)$ mod $d$. Since the number of determinant $1$ automorphisms of the ternary form $x^2+y^2+z^2$ is $24$, we obtain the Lemma.
\end{proof}

\begin{lem}\label{note13}[Note 13 of \cite{linnik-quaternions}] Let $\phi(x,y) = px^2+2qxy+ry^2$ be a quadratic form with determinant $d = pr-q^2$, and put $\delta = (p,q,r)$. The number of representations of $\phi$ as a sum of three squares of linear forms
\[
\phi(x,y) = (a_1x+b_1y)^2 + (a_2x+b_2y)^2 + (a_3x+b_3y)^2
\]
satisfying
\begin{itemize}
\item[\rm a)] $(a_1,a_2,a_3) = 1$,
\item[\rm b)] $\delta \mid (a_2b_3-a_3b_2, a_3b_1-a_1b_3, a_1b_2-a_2b_1)$,
\end{itemize}
is at most
\[
96\cdot 2^{\omega(p)}\tau\left(\frac{d}{\delta^2}\right)\left(\frac{d}{\delta^2},p\right)^{[\frac{1}{2}]}.
\]
\end{lem}
\begin{proof} We will follow Maly\v{s}ev's proof from \cite{linnik-quaternions}, making the bounds a bit more precise as we go (note that Maly\v{s}ev incorrectly claims that we may assume $(e,\frac{p}{\delta}) = 1$ by performing a change of variables, but after such a change of variables there is no guarantee that condition a) is still satisfied). Set $e = (a_2b_3-a_3b_2, a_3b_1-a_1b_3, a_1b_2-a_2b_1)$, and write $\phi(x,y) = \delta(\alpha x^2 + 2\beta xy + \gamma y^2)$. We have
\[
d = (a_2b_3-a_3b_2)^2 + (a_3b_1-a_1b_3)^2 + (a_1b_2-a_2b_1)^2,
\]
so $e^2\mid d$. Since $(a_1,a_2,a_3) = 1$ and $a_2b_3-a_3b_2 \equiv a_3b_1-a_1b_3 \equiv a_1b_2-a_2b_1 \equiv 0 \pmod{e}$, we have
\[
b_i \equiv \lambda a_i \pmod{e}
\]
for some $0 \le \lambda < e$. Thus each linear form $a_i(x-\frac{\lambda}{e}y) + b_i(\frac{1}{e}y)$ has integer coefficients, so we can conclude that $\phi(x-\frac{\lambda}{e}y,\frac{1}{e}y)$ has integer coefficients (with even coefficients on the cross terms) and is \emph{properly} represented as a sum of three squares of linear forms. Since it has integer coefficients, we have the congruences
\begin{align*}
\alpha \lambda - \beta &\equiv 0 \pmod{\frac{e}{\delta}},\\
(\alpha \lambda - \beta)^2 &\equiv -\frac{d}{\delta^2} \pmod{\alpha\frac{e^2}{\delta}}.
\end{align*}
If we write $\alpha\lambda - \beta = \frac{e}{\delta}t$, we find that $t^2 \equiv -\frac{d}{e^2}\pmod{\alpha\delta}$, and there is at most one value of $\lambda \pmod{e}$ corresponding to a given value of $t\pmod{\alpha\delta}$. Since $\alpha\delta = p$, we can apply Lemma \ref{square-roots} to see that there are at most
\[
2^{1+\omega\left(\frac{\displaystyle p}{\left(\frac{d}{e^2},p\right)}\right)}\left(\frac{d}{e^2},p\right)^{[\frac{1}{2}]} \le 2^{1+\omega(p)}\left(\frac{d}{\delta^2},p\right)^{[\frac{1}{2}]}
\]
choices of $\lambda$ for a given choice of $e$. For each choice of $e$ and $\lambda$, there are at most
\[
48\cdot 2^{\omega(\frac{d}{e^2})}
\]
proper representations of $\phi(x-\frac{\lambda}{e}y,\frac{1}{e}y)$ as a sum of three squares of linear forms by Lemma \ref{gauss-ternary}. Since the sum of $2^{\omega(\frac{d}{e^2})}$ over all $e$ such that $\delta\mid e$ and $e^2\mid d$ is $\tau\left(\frac{d}{\delta^2}\right)$, we are done.
\end{proof}

\begin{defn} If $L_1, L_2$ are two proper integral vectors of the same norm, we define (following \cite{akshay}) $\Lambda_{L_1\rightarrow L_2}$ to be
\[
\Lambda_{L_1\rightarrow L_2} = \{\lambda\in \mathbb{B}\mid L_1\lambda = \lambda L_2\}.
\]
This is a left $\mathbb{Z}[L_1]$-module and a right $\mathbb{Z}[L_2]$-module. Let $A,C$ be a $\mathbb{Z}$-basis of $\Lambda_{L_1\rightarrow L_2}$. Then by Note 8 of Chapter 1 of \cite{linnik-quaternions}, we can write (possibly after changing the sign of $A$)
\[
\bar{A}C = b + L_2
\]
for $b$ an integer. Set $a = \mbox{Nm}(A), c = \mbox{Nm}(C)$. The quadratic form $ax^2 + 2bxy + cy^2$ of determinant $\mbox{Nm}(L_2)$, abbreviated as $(a,b,c)$, is called the \emph{form that directs the angle} $(L_2, L_1)$.
\end{defn}

Linnik \cite{linnik-quaternions} divides up the conjugate pairs based on whether the form directing the angle between them takes on values smaller than $6m^{\frac{1}{2}-\tau}$. If it doesn't take small values, he refers to the form as a ``junior'' form, while if it does take on small values he refers to the form as a ``senior'' form. The hardest part of the argument is finding a good bound on the number of conjugate pairs directed by senior forms. Our treatment of the case of senior forms will differ from Linnik's in that we shall further split into cases in order to achieve an asymptotically better bound.

\begin{lem}[(Bound on the number of conjugate pairs directed by junior forms)]\label{junior} Let $0 < \tau < \frac{1}{2}$ be a fixed real number. Let $s$ be an integer such that $m^{\frac{1}{2}+\tau} \le 5^s < 5m^{\frac{1}{2}+\tau}$. Let $l$ be an integer such that $5^s\mid\mid l^2+m$. Suppose that we have
\begin{align*}
l + L_0 &= QU_0\\
l + L_1 &= \bar{Q}U_1\\
l + L_2 &= \bar{Q}U_2
\end{align*}
where each $L_i$ is a proper integral vector of norm $m$, $Q$ is a proper integral quaternion of norm $5^s$, and the quaternions $U_i$ are integral quaternions. Let $(a_1,b_1,c_1)$ and $(a_2,b_2,c_2)$ be reduced quadratic forms (satisfying $2|b_i| \le c_i \le a_i$) directing the angles $(L_0,L_1)$ and $(L_0,L_2)$, respectively. If $c_i \ge 6m^{\frac{1}{2}-\tau}$ and $m^{\tau} > \frac{8\sqrt{3}}{3}$, then $L_1 = L_2$.
\end{lem}
\begin{proof} Let $A_i, C_i$ be the corresponding basis of $\Lambda_{L_i\rightarrow L_0}$. The strategy is to show that the plane spanned by $A_1, C_1$ coincides with the plane spanned by $A_2, C_2$. In order to do so, we first show that $\mbox{Tr}(A_iQ) = \mbox{Tr}(C_iQ) = 0$, so the two planes lie in a common hyperplane. To show that these traces are zero, we will prove that they are both congruent to zero modulo $5^s$ and smaller than $5^s$.

We have
\[
A_iQU_0 = A_i(l+L_0) = (l+L_i)A_i = \bar{Q}U_iA_i,
\]
so $\bar{Q}$ left divides $A_iQU_0$. Since $5$ doesn't divide the norm of $U_0$, $\bar{Q}$ must left divide $A_iQ$, so $Q$ right divides both $A_iQ$ and its conjugate. Thus the proper quaternion $Q$ is a right divisor of the integer $\mbox{Tr}(A_iQ)$, so we must have $\mbox{Tr}(A_iQ) \equiv 0 \pmod{5^s}$. Similarly we have $\mbox{Tr}(C_iQ) \equiv 0 \pmod{5^s}$.

Since $c_i \le \sqrt{\frac{4}{3}m}$, we have
\[
\mbox{Tr}(C_iQ)^2 \le 4\mbox{Nm}(C_iQ) = 4\cdot 5^sc_i \le 4\cdot 5^s\sqrt{\frac{4}{3}m} = \frac{8\sqrt{3}}{3}5^{s}m^{\frac{1}{2}} < 5^{2s},
\]
since $m^{\tau} > \frac{8\sqrt{3}}{3}$. Thus $\mbox{Tr}(C_iQ) = 0$. Note that so far we have not used the lower bound on $c_i$.

By the lower bound on $c_i$ and the fact that the determinant of the quadratic form $(a_i,b_i,c_i)$ is $m$ we have
\[
a_i = \frac{m+b_i^2}{c_i} \le \frac{1}{6}m^{\frac{1}{2}+\tau} + \frac{1}{4}c_i \le \frac{11}{48}m^{\frac{1}{2}+\tau}.
\]
Thus, we have
\[
\mbox{Tr}(A_iQ)^2 \le 4\cdot 5^sa_i \le \frac{11}{12}5^{s}m^{\frac{1}{2}+\tau} < 5^{2s},
\]
so $\mbox{Tr}(A_iQ) = 0$.

From the above we see that $QA_i, QC_i$ are vectors. Thus,
\[
(QA_i)(QC_i) = (-\bar{A_i}\bar{Q})(QC_i) = -5^s\bar{A_i}C_i = -5^sb-5^sL_0,
\]
so $QA_i, QC_i$ are orthogonal to $L_0$. Thus $A_2$ is in the same plane as $A_1,C_1$, so $A_2 \in \Lambda_{L_1\rightarrow L_0}$, and we see $L_2 = A_2L_0A_2^{-1} = L_1$.
\end{proof}

\begin{lem}[(Bound on the number of conjugate pairs directed by senior forms)]\label{senior} Let $0 < \tau < \frac{1}{2}$ be a fixed real number. Let $s$ be an integer such that $m^{\frac{1}{2}+\tau} \le 5^s < 5m^{\frac{1}{2}+\tau}$. Let $l$ be an integer such that $5^s\mid\mid l^2+m$. Then the number of pairs $(L_1,L_2)$ such that we have
\begin{align*}
l + L_1 &= QU_1,\\
l + L_2 &= \bar{Q}U_2,
\end{align*}
where each $L_i$ is a proper integral vector of norm $m$, $Q$ is a proper integral quaternion of norm $5^s$ which is in standard form, and the quaternions $U_i$ are integral quaternions, such that there is an element $C\in \Lambda_{L_2\rightarrow L_1}$ with norm $c$ satisfying $c < 6m^{\frac{1}{2}-\tau}$, is at most
\[
t(m)+96S,
\]
where $S$ is the sum
\[
S = \sum_{c=1}^{6m^{\frac{1}{2}-\tau}}\sum_{5^{s_1}\mid c}\sum_{d=1}^{5^{s_1}\sqrt{\frac{4m}{5^sc}}}2^{\omega(5c)}\tau(4m-5^{s-2s_1}d^2c)\left(4m, c\right)^{[\frac{1}{2}]},
\]
if $m^{\tau} > \frac{8\sqrt{3}}{3}$ and $5\nmid m$.
\end{lem}
\begin{proof} By possibly negating $C$, we can assume that $\mbox{Tr}(U_1\bar{C}) \ge 0$, and by dividing $C$ by an integer we can assume that $C$ is proper. From the proof of Lemma \ref{junior}, $CQ$ is a vector. We divide into two cases, based on whether $\mbox{Tr}(U_1\bar{C}) = 0$ or not. If $\mbox{Tr}(U_1\bar{C}) = 0$, then since $CQ$ and $U_1\bar{C}$ are both vectors we have
\[
l+L_2 = C(l+L_1)C^{-1} = \frac{1}{c}CQU_1\bar{C} = \frac{1}{c}(-\bar{Q}\bar{C})(-C\bar{U_1}) = \bar{Q}\bar{U_1},
\]
so $L_2$ is determined by $L_1$, and the number of such pairs is at most $t(rm)$.

Now we count pairs $(L_1, L_2)$ with $\mbox{Tr}(U_1\bar{C}) > 0$. If $CQ$ is not proper, then by an easy application of B\'{e}zout's identity we see that we can write
\[
C = C'T,\ Q = \bar{T}Q',
\]
where $T$ is integral, $C'Q'$ is proper, and $Q'$ is in standard form. Let $\mbox{Nm}(T) = 5^{s_1}$, let $\mbox{Nm}(C') = c'$. We have $0 \le s_1 < s$. We now count the number of such pairs $(L_1,L_2)$ for fixed values of $s_1, c$.

First we show that there are not too many possible values of $\mbox{Tr}(U_1\bar{C})$, by showing that it must be congruent to $0$ modulo $c'$. We just need to show that $C'$ left divides $U_1\bar{C}$, or equivalently that $c'$ divides $\bar{C'}U_1\bar{C}$.

Let $L_1' = TL_1T^{-1}$, and let $U_1' = U_1\bar{T}$. Note that with these definitions we have $U_1\bar{C} = U_1'\bar{C'}$ and $C'L_1'C'^{-1} = L_2$. We have $L_1'^2 = -m$ and
\[
l + L_1' = Q'U_1\bar{T} = Q'U_1',
\]
so $L_1'$ is an integral vector, and it is proper since $5\nmid \mbox{Nm}(L_1')$, $L_1$ is proper, and $L_1 = \frac{\bar{T}L_1'T}{5^{s_1}}$. Since $CQ$ is a vector, so is $C'Q'$. From this we see
\[
\bar{Q'}\bar{C'}U_1'\bar{C'} = -C'Q'U_1'\bar{C'} = -C'(l+L_1')\bar{C'} = -c'(l+L_2).
\]
Since $C'Q'$ is proper, $\bar{Q'}$ and $C'$ have no common right divisors, so there exist $X,Y$ such that $X\bar{Q'} + YC' = 1$. Thus
\[
\bar{C'}U_1'\bar{C'} = (X\bar{Q'}+YC')\bar{C'}U_1'\bar{C'} = c'(-X(l+L_2)+YU_1'\bar{C'}),
\]
so $U_1'\bar{C'} = C'(-X(l+L_2)+YU_1'\bar{C'})$. Thus $\bar{C'}$ is a right divisor of the integer $\mbox{Tr}(U_1'\bar{C'})$, so
\[
\mbox{Tr}(U_1'\bar{C'}) \equiv 0 \pmod{c'}.
\]

Define $H_1, H_2, a_i, b_i, d$ by
\begin{align*}
H_1 &= 5^{-s_1}CQ = C'Q' = a_1i + a_2j + a_3k,\\
H_2 &= 2U_1\bar{C} = 2U_1'\bar{C'} = dc' + b_1i + b_2j + b_3k.
\end{align*}
We next show that for fixed values of $s_1, c, d$ there are not too many possible values for $a_i, b_i$.

Consider the binary quadratic form
\[
\mbox{Nm}(-x\bar{H_1} + yH_2) = \mbox{Nm}(H_1)x^2 - \mbox{Tr}(H_1H_2)xy + \mbox{Nm}(H_2)y^2.
\]
We have $\mbox{Nm}(H_1) = 5^{s-s_1}c'$, $\mbox{Nm}(H_2) = 4c'u'$ where $u' = \mbox{Nm}(U_1')$. Also, $\mbox{Tr}(H_1H_2) = 2c'\mbox{Tr}(Q'U_1') = 4c'l$. Thus, we have
\begin{align*}
\mbox{Nm}(-x\bar{H_1} + yH_2) &= 5^{s-s_1}c'x^2 - 4c'lxy + 4c'u'y^2\\
&= (dc')^2y^2 + (a_1x + b_1y)^2 + (a_2x + b_2y)^2 + (a_3x + b_3y)^2,
\end{align*}
so
\[
5^{s-s_1}c'x^2 - 4c'lxy + (4c'u'-d^2c'^2)y^2 = (a_1x + b_1y)^2 + (a_2x + b_2y)^2 + (a_3x + b_3y)^2.
\]

Let $\phi(x,y)$ be defined by
\[
\phi(x,y) = 5^{s-s_1}c'x^2 - 4c'lxy + (4c'u'-d^2c'^2)y^2.
\]
Since $5\nmid l$, we have $c' = (5^{s-s_1}c', -2c'l, 4c'u'-d^2c'^2)$. Note that $(a_1,a_2,a_3) = 1$ since $C'Q'$ is proper, and
\[
2c'U_1'Q' = H_2H_1 = dc'H_1 + \mbox{Re}(H_1H_2) + (a_2b_3-a_3b_2)i + (a_3b_1-a_1b_3)j + (a_1b_2-a_2b_1)k,
\]
so $c'\mid (a_2b_3-a_3b_2, a_3b_1-a_1b_3, a_1b_2-a_2b_1)$. Since $5^{s-s_1}u' = \mbox{Nm}(Q'U_1') = l^2 + m$, the determinant of $\phi$ is
\[
c'^2(4\cdot 5^{s-s_1}u'-4l^2-5^{s-s_1}d^2c') = c'^2(4m-5^{s-s_1}d^2c').
\]

Now we are able to apply Lemma \ref{note13} to see that for fixed values of $s_1, c, d$ there are at most
\begin{align*}
&96\cdot 2^{\omega(5^{s-s_1}c')}\tau(4m-5^{s-s_1}d^2c')\left(4m-5^{s-s_1}d^2c', 5^{s-s_1}c'\right)^{[\frac{1}{2}]}\\
&= 96\cdot 2^{\omega(5c)}\tau(4m-5^{s-s_1}d^2c')\left(4m, c\right)^{[\frac{1}{2}]}
\end{align*}
tuples $(a_1, a_2, a_3, b_1, b_2, b_3)$ satisfying the above conditions.

Each such tuple determines $H_1, H_2$, and from $H_1 = C'Q'$ we can uniquely determine $Q'$ since $H_1$ is proper and $Q'$ is in standard form. From this we can determine $C'$, then $U_1'$ by $H_2 = 2U_1'\bar{C'}$, then $L_1'$ by $l+L_1' = Q'U_1'$, then $L_2$ by $C'L_1'C'^{-1} = L_2$. Since $U_1' = U_1\bar{T}$, $5\nmid \mbox{Nm}(U_1)$, and $\bar{T}$ is in standard form, we can uniquely determine $\bar{T}$ and this determines $L_1$ by $TL_1T^{-1} = L_1'$. Thus, each such tuple corresponds to at most one pair $(L_1,L_2)$.

Finally, since $\phi$ is nonnegative, we have $\det \phi \ge 0$, so
\[
0 < d \le \sqrt{\frac{4rm}{5^{s-s_1}c'}} = 5^{s_1}\sqrt{\frac{4rm}{5^sc}}.\qedhere
\]
\end{proof}

In order to obtain a good bound on the sum $S$ occurring in Lemma \ref{senior}, we will make use of the following level lowering trick (this type of trick for handling sums involving divisor functions is originally due to van der Corput \cite{vandercorput}, Wolke \cite{wolke}, and Landreau \cite{landreau}).

\begin{lem}\label{level} For any integer $n$, we have
\[
\tau(n) \le \left(\frac{16}{3}\right)^{13} \sum_{d\mid n,\ d\le n^{\frac{3}{16}}} \tau(d)^2.
\]
\end{lem}
\begin{proof} We follow Munshi's argument from \cite{munshi}, making some modifications to handle the case where $n$ is not squarefree. Write $n = \prod_{i=1}^{\Omega(n)}p_i$, with each $p_i$ prime. For any $d$ let $\langle\substack{n\\d}\rangle$ be the number of subsets $S\subseteq \{1,...,\Omega(n)\}$ with $\prod_{i\in S}p_i = d$. Thus $\langle\substack{n\\d}\rangle$ is a multiplicative function of $n$ and $d$, and for $p$ a prime we have $\langle\substack{p^e\\p^f}\rangle = \binom{e}{f}$.

Interpreting $16^{\Omega(n)}$ as the number of partitions of the set $\{1, ..., \Omega(n)\}$ into $16$ subsets $S_1, ..., S_{16}$, and writing $d = \prod_{i\le 3}\prod_{j\in S_i}p_j$, we have the equality
\[
16^{\Omega(n)} = \sum_{d\mid n}\langle\substack{n\\d}\rangle 3^{\Omega(d)}13^{\Omega(n/d)}.
\]
Now we apply a special case of the {M}anickam--{M}ikl\'os--{S}inghi conjecture. For any integers $n \ge k$, define $g(n,k)$ by
\[
g(n,k) = \min_{a_1 + \cdots + a_n = 0} \big|\big\{S\subseteq \{1,...,n\}\big| |S| = k, \sum_{i\in S}a_i \ge 0\big\}\big|.
\]
The {M}anickam--{M}ikl\'os--{S}inghi conjecture states that for $n \ge 4k$, $g(n,k) = \binom{n-1}{k-1}$. In \cite{mms}, this is verified for $k \le 7$. We will apply this with $n = 16, k = 3$.

For any partition of $\{1, ..., \Omega(n)\}$ into $16$ subsets $S_1, ..., S_{16}$, set $a_i = \frac{1}{16}\log(n) - \sum_{j\in S_i}\log(p_j)$. Then since $\sum_{i=1}^{16}a_i = 0$, the number of subsets $I\subset \{1,...,16\}$ of size $3$ such that $\sum_{i\in I}a_i \ge 0$, or in other words such that $\prod_{i\in I}\prod_{j\in S_i}p_j \le n^{\frac{3}{16}}$, is at least $g(16,3) = \binom{15}{2}$. Since $\binom{15}{2} = \frac{3}{16}\binom{16}{3}$, this show that
\[
16^{\Omega(n)} \le \frac{16}{3}\sum_{d\mid n,\ d\le n^{\frac{3}{16}}}\langle\substack{n\\d}\rangle 3^{\Omega(d)}13^{\Omega(n/d)}.
\]
By H\"{o}lder's inequality, we see that
\[
\sum_{d\mid n,\ d\le n^{\frac{3}{16}}}\langle\substack{n\\d}\rangle 3^{\Omega(d)}13^{\Omega(n/d)} \le \Bigg(\sum_{d\mid n,\ d\le n^{\frac{3}{16}}} \tau(d)^2\Bigg)^{\frac{1}{13}}\left(\sum_{d\mid n}\left(\langle\substack{n\\d}\rangle 3^{\Omega(d)}13^{\Omega(n/d)}\right)^{\frac{13}{12}}\frac{1}{\tau(d)^{\frac{2}{12}}}\right)^{\frac{12}{13}}.
\]
Putting the last two inequalities together, we have
\[
\left(\sum_{d\mid n}\left(\langle\substack{n\\d}\rangle \frac{3^{\Omega(d)}13^{\Omega(n/d)}}{16^{\Omega(n)}}\right)^{\frac{13}{12}}\frac{1}{\tau(d)^{\frac{2}{12}}}\right)^{-12} \le \left(\frac{16}{3}\right)^{13} \sum_{d\mid n,\ d\le n^{\frac{3}{16}}} \tau(d)^2.
\]
We would like to show that the left hand side is at least as large as $\tau(n)$. Since the left hand side is a multiplicative function of $n$, it's enough to check this when $n$ is a prime power, say $n = p^e$. In other words, we need to check that for $e \ge 0$ we have
\[
e+1 \le \left(\sum_{f\le e}\left(\binom{e}{f} \frac{3^{f}13^{e-f}}{16^{e}}\right)^{\frac{13}{12}}\frac{1}{(f+1)^{\frac{2}{12}}}\right)^{-12}.
\]
For $e \le 27$, we can check this by a straightforward computation. For $e\ge 28$, this follows from
\begin{align*}
&\left(\sum_{f\le e}\left(\binom{e}{f} \frac{3^{f}13^{e-f}}{16^{e}}\right)^{\frac{13}{12}}\frac{1}{(f+1)^{\frac{2}{12}}}\right)^{-12}\\
\ge &\left(\sum_{f\le e}\binom{e}{f} \frac{3^{f}13^{e-f}}{16^{e}}\frac{1}{(f+1)^{\frac{2}{12}}}\right)^{-12}\\
= &\left(\frac{16}{3(e+1)}\sum_{f+1\le e+1}\binom{e+1}{f+1} \frac{3^{f+1}13^{e-f}}{16^{e+1}}(f+1)^{\frac{10}{12}}\right)^{-12}\\
\ge &\left(\frac{16}{3(e+1)}\left(\frac{3}{16}(e+1)\right)^{\frac{10}{12}}\right)^{-12} = \left(\frac{3(e+1)}{16}\right)^2,
\end{align*}
where the first inequality follows from the fact that $\binom{e}{f} \frac{3^{f}13^{e-f}}{16^{e}} \le 1$, and the second inequality follows from Jensen's inequality applied to the concave function $x \mapsto x^{\frac{10}{12}}$.
\end{proof}

In the proof of the next lemma, we will also make free use of the easy inequalities
\begin{align*}
\sum_{x\le n} \frac{\tau(x)^k}{x} &\le \left(\sum_{x\le n} \frac{1}{x}\right)^{2^k} \le \log(en)^{2^k},\\
\sum_{x\le n} \tau(x)^k &\le \sum_{d_1, ..., d_{2^k-1} \le n} \left\lfloor\frac{n}{d_1 \cdots d_{2^k-1}}\right\rfloor \le n\log(en)^{2^k-1}.
\end{align*}

\begin{lem}\label{sum} For $0 < \tau < \frac{1}{2}$, $m^{\frac{1}{2}+\tau} \le 5^s < 5m^{\frac{1}{2}+\tau}$, $m^{\tau} > \frac{8\sqrt{3}}{3}$, $5\nmid m$, the sum
\[
S = \sum_{c=1}^{6m^{\frac{1}{2}-\tau}}\sum_{5^{s_1}\mid c}\sum_{d=1}^{5^{s_1}\sqrt{\frac{4m}{5^sc}}}2^{\omega(5c)}\tau(4m-5^{s-2s_1}d^2c)\left(4m, c\right)^{[\frac{1}{2}]}
\]
is at most
\[
\left(e^{12}F_1(4m)\log(e^{7}m)+e^{31}F_2(4m)\right)\log(e^{12}m)^8m^{\frac{1}{2}-\tau}+e^{13667}\log(e^{12}m)^6m^{\frac{1}{2}-\frac{1}{160}},
\]
where $F_1, F_2$ are multiplicative functions given by
\begin{align*}
F_1(n) &= \sum_{gh^2\mid n} \frac{\tau(g)\tau(h^2)^22^{\omega(g)+\omega(h)}g^{[\frac{1}{2}]}}{hg},\\
F_2(n) &= \sum_{gh^2\mid n} \frac{\tau(g)\tau(h^2)^22^{\omega(g)}g^{[\frac{1}{2}]}}{hg}.
\end{align*}
$F_1,F_2$ satisfy the inequalities
\[
F_1(4m) \le e^{62}m^{10^{-6}},\ F_2(4m) \le e^{46}m^{10^{-6}}.
\]
\end{lem}
\begin{proof} Write $c = 5^{s_1}gk$, with $g = (4m,c)$. Then we see that $S$ is at most
\[
2\sum_{\substack{g\mid 4m\\s_1<s}}2^{\omega(g)}\tau(g)g^{[\frac{1}{2}]}\sum_{\substack{(k,\frac{4m}{g})=1\\5^{s_1}gk\le 6m^{\frac{1}{2}-\tau}}}\sum_{d^2\le\frac{4m}{5^{s-s_1}gk}}2^{\omega(k)}\tau\left(\frac{4m}{g}-5^{s-s_1}d^2k\right).
\]
Fix values of $g$ and $s_1$. Let $\mathcal{A}$ be the set of ordered pairs of positive integers $(k,d)$ such that $(k,\frac{4m}{g}) = 1$, $5^{s_1}gk\le 6m^{\frac{1}{2}-\tau}$, and $d^2\le\frac{4m}{5^{s-s_1}gk}$. By Lemma \ref{level}, we have
\[
\sum_{(k,d)\in \mathcal{A}} 2^{\omega(k)}\tau\left(\frac{4m}{g}-5^{s-s_1}d^2k\right) \le \left(\frac{16}{3}\right)^{13}\sum_{f\le \left(\frac{4m}{g}\right)^{\frac{3}{16}}}\tau(f)^2\sum_{\substack{(k,d)\in \mathcal{A}\\f\mid \frac{4m}{g}-5^{s-s_1}d^2k}} 2^{\omega(k)}.
\]
We divide the inner sum into two sums $S_1, S_2$ based on whether $d \ge f$ or $d<f$:
\begin{align*}
S_1 &= \sum_{\substack{(k,d)\in \mathcal{A},\ d\ge f\\f\mid \frac{4m}{g}-5^{s-s_1}d^2k}} 2^{\omega(k)},\\
S_2 &= \sum_{\substack{(k,d)\in \mathcal{A},\ d < f\\f\mid \frac{4m}{g}-5^{s-s_1}d^2k}} 2^{\omega(k)}.
\end{align*}
For fixed values of $f$ and $k$, the number of integers $d\ge f$ such that $(k,d)\in \mathcal{A}$ and $f\mid \frac{4m}{g}-5^{s-s_1}d^2k$ is at most
\[
2^{1+\omega(f)}\left(f,\frac{4m}{g}\right)^{[\frac{1}{2}]}\cdot\frac{1}{f}\sqrt{\frac{4m}{5^{s-s_1}gk}}
\]
by Lemma \ref{square-roots} and the fact that $5^{s-s_1}k$ is relatively prime to $\frac{4m}{g}$. Since $\frac{1}{x}\sum_{k\le x}2^{\omega(k)} \le \log(ex)$ and $\frac{1}{\sqrt{k}}$ is decreasing in $k$, with $\sum_{k\le x}\frac{1}{\sqrt{k}} < 2\sqrt{x}$, we have
\begin{align*}
S_1 &\le 2^{1+\omega(f)}\frac{\left(f,\frac{4m}{g}\right)^{[\frac{1}{2}]}}{f}\sum_{k \le \frac{6(m)^{\frac{1}{2}-\tau}}{5^{s_1}g}} \sqrt{\frac{4m}{5^{s-s_1}gk}}2^{\omega(k)}\\
&<\ 2^{1+\omega(f)}\frac{\left(f,\frac{4m}{g}\right)^{[\frac{1}{2}]}}{f}\sqrt{\frac{4m}{5^{s-s_1}g}}\cdot 2\sqrt{\frac{6m^{\frac{1}{2}-\tau}}{5^{s_1}g}}\log\left(e\frac{6m^{\frac{1}{2}-\tau}}{5^{s_1}g}\right)\\
&<\ 4\sqrt{6}\cdot 2^{\omega(f)}\frac{\left(f,\frac{4m}{g}\right)^{[\frac{1}{2}]}}{f}\frac{m^{\frac{1}{2}-\tau}}{g}\log(e^3m).
\end{align*}

Now we consider the sum $S_2$. Note that since $\left(\frac{4m}{g},5^{s-s_1}k\right) = 1$, in order to have $f\mid \frac{4m}{g}-5^{s-s_1}d^2k$ we must have $\left(f,\frac{4m}{g}\right)\mid d^2$, and the number of such $d < f$ is at most $\frac{\left(f,\frac{4m}{g}\right)^{[\frac{1}{2}]}}{\left(f,\frac{4m}{g}\right)}f$. Fix such a $d$, and define $K(d)$ by the formula
\[
K(d) = \min\left(\frac{6m^{\frac{1}{2}-\tau}}{5^{s_1}g}, \frac{4m}{5^{s-s_1}gd^2}\right).
\]
Let $f' = \frac{f}{\left(f,\frac{4m}{g}\right)}$, and let $k_0$ be the congruence class modulo $f'$ satisfying $f\mid \frac{4m}{g}-5^{s-s_1}d^2k_0$, assuming one exists. By Lemma \ref{level}, we have
\begin{align*}
\sum_{\substack{k\le K(d)\\k\equiv k_0\pmod{f'}}} 2^{\omega(k)}\le &\left(\frac{16}{3}\right)^{13}\sum_{\substack{k\le K(d)\\k\equiv k_0\pmod{f'}}} \sum_{\substack{x\mid k\\x\le K(d)^{\frac{3}{16}}}}\tau(x)^2\\
\le &\left(\frac{16}{3}\right)^{13}\sum_{x\le K(d)^{\frac{3}{16}}}\tau(x)^2\left(\frac{K(d)}{xf'}+1\right)\\
\le &\left(\frac{16}{3}\right)^{13}\left(\frac{K(d)}{f'}\log\left(eK(d)^{\frac{3}{16}}\right)^4+K(d)^{\frac{3}{16}}\log\left(eK(d)^{\frac{3}{16}}\right)^3\right)\\
\le &\frac{2^{32}}{3^{9}}\log(e^{12}m)^4K(d)\frac{\left(f,\frac{4m}{g}\right)}{f}+\frac{2^{37}}{3^{10}}\log(e^{12}m)^3\left(\frac{6m^{\frac{1}{2}-\tau}}{5^{s_1}g}\right)^{\frac{3}{16}}.
\end{align*}
Since $K(d)$ is decreasing in $d$, we have
\[
\sum_{\left(f,\frac{4m}{g}\right)\mid d^2} K(d) \le \frac{\left(f,\frac{4m}{g}\right)^{[\frac{1}{2}]}}{\left(f,\frac{4m}{g}\right)}\int_{d\ge 0} K(d) \le 4\sqrt{6}\cdot\frac{\left(f,\frac{4m}{g}\right)^{[\frac{1}{2}]}}{\left(f,\frac{4m}{g}\right)}\frac{m^{\frac{1}{2}-\tau}}{g},
\]
so
\begin{align*}
S_2 &\le \sum_{\left(f,\frac{4m}{g}\right)\mid d^2,\ d < f} \left(\frac{2^{32}}{3^{9}}\log(e^{12}m)^4K(d)\frac{\left(f,\frac{4m}{g}\right)}{f}+\frac{2^{37}}{3^{10}}\log(e^{12}m)^3\left(\frac{6m^{\frac{1}{2}-\tau}}{5^{s_1}g}\right)^{\frac{3}{16}}\right)\\
&\le \frac{2^{34}\sqrt{6}}{3^{9}}\log(e^{12}m)^4\frac{\left(f,\frac{4m}{g}\right)^{[\frac{1}{2}]}}{f}\frac{m^{\frac{1}{2}-\tau}}{g} + \frac{2^{37}}{3^{10}}\log(e^{12}m)^3f\left(\frac{6m^{\frac{1}{2}-\tau}}{5^{s_1}g}\right)^{\frac{3}{16}}.
\end{align*}

Putting the bounds for $S_1, S_2$ together, we have
\begin{align*}
\left(\frac{16}{3}\right)^{13}\sum_{f\le \left(\frac{4m}{g}\right)^{\frac{2}{11}}}\tau(f)^2\sum_{\substack{(k,d)\in \mathcal{A}\\f\mid \frac{4m}{g}-5^{s-s_1}d^2k}} 2^{\omega(k)} \le A_1 + A_2 + A_3,
\end{align*}
where
\begin{align*}
A_1 &= \frac{2^{54}\sqrt{6}}{3^{13}}\log(e^3m)\sum_{f\le \left(\frac{4m}{g}\right)^{\frac{3}{16}}}\tau(f)^22^{\omega(f)}\frac{\left(f,\frac{4m}{g}\right)^{[\frac{1}{2}]}}{f}\frac{m^{\frac{1}{2}-\tau}}{g},\\
A_2 &= \frac{2^{86}\sqrt{6}}{3^{22}}\log(e^{12}m)^4\sum_{f\le \left(\frac{4m}{g}\right)^{\frac{3}{16}}}\tau(f)^2\frac{\left(f,\frac{4m}{g}\right)^{[\frac{1}{2}]}}{f}\frac{(m)^{\frac{1}{2}-\tau}}{g},\\
A_3 &= \frac{2^{89}}{3^{23}}\log(e^{12}m)^3\sum_{f\le \left(\frac{4m}{g}\right)^{\frac{3}{16}}}\tau(f)^2f\left(\frac{6m^{\frac{1}{2}-\tau}}{5^{s_1}g}\right)^{\frac{3}{16}}.
\end{align*}
In order to estimate the sums occurring in $A_1, A_2$, we split the sum over possible values of $h = \left(f,\frac{4m}{g}\right)^{[\frac{1}{2}]}$. Setting $f' = \frac{f}{h^2}$, we have
\begin{align*}
\sum_{f\le \left(\frac{4m}{g}\right)^{\frac{3}{16}}}\tau(f)^22^{\omega(f)}\frac{\left(f,\frac{4m}{g}\right)^{[\frac{1}{2}]}}{f} \le& \sum_{h^2\mid \frac{4m}{g}} \frac{\tau(h^2)^22^{\omega(h)}}{h} \sum_{f' \le \left(\frac{4m}{g}\right)^{\frac{3}{16}}} \frac{\tau(f')^22^{\omega(f')}}{f'}\\
\le& \left(\frac{3}{16}\right)^8\log(e^{7}m)^8\sum_{h^2\mid \frac{4m}{g}} \frac{\tau(h^2)^22^{\omega(h)}}{h},
\end{align*}
and similarly we have
\[
\sum_{f\le \left(\frac{4m}{g}\right)^{\frac{3}{16}}}\tau(f)^2\frac{\left(f,\frac{4m}{g}\right)^{[\frac{1}{2}]}}{f} \le \left(\frac{3}{16}\right)^4\log(e^{7}m)^4\sum_{h^2\mid \frac{4m}{g}} \frac{\tau(h^2)^2}{h}.
\]
The sum in $A_3$ is bounded by
\[
\sum_{f\le \left(\frac{4m}{g}\right)^{\frac{3}{16}}}\tau(f)^2f\left(\frac{6m^{\frac{1}{2}-\tau}}{5^{s_1}g}\right)^{\frac{3}{16}} \le \left(\frac{3}{16}\right)^3\log(e^{7}m)^3\frac{(4m)^{\frac{15}{32}}}{g^{\frac{9}{16}}}.
\]

Multiplying $A_1+A_2+A_3$ by $2\tau(g)2^{\omega(g)}g^{[\frac{1}{2}]}$ and summing over $g$, we see that
\begin{align*}
S \le\ &\frac{2^{23}\sqrt{6}}{3^{5}}\log(e^{7}m)^9m^{\frac{1}{2}-\tau}\sum_{gh^2\mid 4m} \frac{\tau(g)\tau(h^2)^22^{\omega(g)+\omega(h)}g^{[\frac{1}{2}]}}{hg}\\
&+\frac{2^{71}\sqrt{6}}{3^{18}}\log(e^{12}m)^8m^{\frac{1}{2}-\tau}\sum_{gh^2\mid 4m} \frac{\tau(g)\tau(h^2)^22^{\omega(g)}g^{[\frac{1}{2}]}}{hg}\\
&+\frac{2^{79}}{3^{20}}\log(e^{12}m)^6m^{\frac{1}{2}-\frac{1}{32}}\sum_{g\mid 4m}\frac{\tau(g)2^{\omega(g)}g^{[\frac{1}{2}]}}{g^{\frac{9}{16}}}.
\end{align*}
We can bound the last sum above as follows: for every prime $p$, we can find the exponent $k$ which maximizes
\[
p^{-\frac{k}{40}}\sum_{g\mid p^k}\frac{\tau(g)2^{\omega(g)}g^{[\frac{1}{2}]}}{g^{\frac{9}{16}}}.
\]
For $p > 1500000$, the maximum value of the above sum is $1$. By simply multiplying these maximum values together over all primes $p \le 1500000$, we see that
\[
\sum_{g\mid 4m}\frac{\tau(g)2^{\omega(g)}g^{[\frac{1}{2}]}}{g^{\frac{9}{16}}} < e^{13634}m^{\frac{1}{40}}.
\]
We get the bounds on $F_1, F_2$ in the same way.

Putting everything together, we see that $S$ is at most
\[
\left(e^{12}F_1(4m)\log(e^{7}m)+e^{31}F_2(4m)\right)\log(e^{12}m)^8m^{\frac{1}{2}-\tau}+e^{13667}\log(e^{12}m)^6m^{\frac{1}{2}-\frac{1}{160}}.\qedhere
\]
\end{proof}

\bigskip

\emph{Acknowledgements.} The author would like to thank Professor Soundararajan for encouragement to work on this problem, and for many useful suggestions, especially the suggestion of using a level-lowering trick to handle the divisor functions coming up in many of the sums. The author would also like to thank an anonymous referee who pointed out that the argument for odd cubic polynomials should generalize to an argument applying to general cubic polynomials.

\bibliographystyle{abbrv}
\bibliography{octahedral}

\end{document}